\documentclass[12pt]{amsart}
\usepackage[margin=1in]{geometry}
\usepackage{mathptmx}
\usepackage{mathtools}
\usepackage{amsthm}
\usepackage{amssymb}
\usepackage[alphabetic]{amsrefs}
\usepackage{graphicx}
\usepackage{tikz}
	\usetikzlibrary{decorations.pathreplacing}
	\usetikzlibrary{patterns}
\usepackage{caption}
\usepackage{eucal}

\newcommand{\cd}{\cdot}
\newcommand{\ra}{\rightarrow}
\newcommand{\pr}{\prime}
\newcommand{\de}{\partial}
\newcommand{\te}{\theta}

\newcommand{\cA}{\mathcal{A}}

\newcommand{\dA}{{\rm d}A}
\newcommand{\cB}{\mathcal{B}}
\newcommand{\cE}{\mathcal{E}}
\newcommand{\cF}{\mathcal{F}}

\newcommand{\C}{\mathbb{C}}

\newcommand{\R}{\mathbb{R}}
\newcommand{\Z}{\mathbb{Z}}
\newcommand{\mZ}{\Z}
\newcommand{\F}{\mathbb{F}}

\newcommand{\be}{\begin{equation}}
\newcommand{\ee}{\end{equation}}
\newcommand{\abs}[1]{\left\lvert #1 \right\rvert}

\newcommand{\lbar}[1]{\overline{#1}}
\newcommand{\id}{\mathrm{id}}

\DeclareMathOperator{\mathdash}{-}
\DeclareMathOperator{\parea}{-area}

\renewcommand{\coprod}{\rotatebox[origin = c]{180}{$\prod$}}

\newcommand{\hookuparrow}{\mathrel{\rotatebox[origin=c]{90}{$\hookrightarrow$}}}

\newtheorem{thm}{Theorem}[subsection]
\newtheorem{conjecture}{Conjecture}[subsection]

\newtheorem{lemma}{Lemma}[subsection]
\newtheorem{prop}{Proposition}[subsection]
\newtheorem{innercustomlm}{Lemma}
\newenvironment{customlm}[1]
  {\renewcommand\theinnercustomlm{#1}\innercustomlm}
  {\endinnercustomlm}
\theoremstyle{definition}
\newtheorem{definition}{Definition}[subsection]
\newtheorem*{definition*}{Definition}
\theoremstyle{remark}
\newtheorem*{note}{Note}

\numberwithin{equation}{subsection}

\begin{document}

\title{Building manifolds from quantum codes}

\author{Michael Freedman}
\address{\hskip-\parindent
	Michael Freedman \\
	Microsoft Research, Station Q, and Department of Mathematics \\
	University of California, Santa Barbara \\
	Santa Barbara, CA 93106 \\
}

\author{Matthew Hastings}
\begin{abstract}
We give a procedure for ``reverse engineering" a closed, simply connected, Riemannian manifold with bounded local geometry from a sparse chain complex over $\Z$.  Applying this procedure to chain complexes obtained by ``lifting" recently developed quantum codes, which correspond to chain complexes over $\Z_2$, we construct the first examples of power law $\Z_2$ systolic freedom.

As a result that may be of independent interest in graph theory, we give an efficient randomized algorithm to construct a weakly fundamental cycle basis for a graph, such that each edge appears only polylogarithmically times in the basis.  We use this result to trivialize the fundamental group of the manifold we construct.
\end{abstract}

\maketitle

\section{Building manifolds from quantum codes}
\subsection{Introduction}
The last fifty years has seen a glorious exploration of the possible ``shapes'' of Riemannian manifolds. Often the new ideas contained a combinatorial element reminiscent of geometric topology and rather far from the study of homogenous spaces in which Riemannian geometry had incubated during the previous fifty years. This paper is in the modern spirit of building and studying Riemannian manifolds from a combinatorial perspective. In this case we use recently discovered quantum codes to build closed Riemannian manifolds with a surprising property: \emph{power law $\Z_2$-systolic freedom}. This discovery in effect reverses the historical flow of information. Twenty years ago one of us helped to create quantum codes exhibiting a weaker \emph{poly(log) $\Z_2$-systolic freedom} \cites{flm02,freedman99}, starting with a certain family of Riemannian 3-manifolds. In retrospect it is not surprising that the most efficient constructions will necessarily have a combinatorial complexity beyond a geometer's intuition, perhaps even requiring computer search, and are better discovered within coding theory and then translated into geometry. This translation is our subject. But before launching into the code class and its translations into manifolds, we review the concept of systolic freedom as pioneered in \cite{gromov96} and \cite{gromov99}.

The \emph{systole} of a closed Riemannian surface is defined to be the length of its shortest essential loop. In 1949 Loewner [L] proved that for the torus $T^2$,
\[
    \frac{\text{systole}^2}{\text{area}} \leq \frac{2}{\sqrt{3}}
\]
with equality (only) for the flat torus modeled on a regular hexagon. The l.h.s.\ is called the systolic ratio (SR). Bounding this quantity for general surfaces is still an area of research \cite{wiki} with a strong number theoretic flavor.

More generally, if $M^d$ is a connected, closed Riemannian manifold of dimension $d$, and $p+q=d$, one can discuss the $(p,q)$-SR of $M$.

\begin{definition}
    The $(p,q)$-Riemannian systolic ratio of $M$, $K\mathdash (p,q)\mathdash \operatorname{SR}_R(M)$ is $\underset{P,Q}{\inf}\ \frac{\text{area}_p(P) \text{area}_q(Q)}{\text{vol}_d(M)}$ where the infimum is taken over all smooth $p(q)$-cycles, $P(Q)$ representing non-trivial homology classes in $H_{p(q)}(M;K)$ where $K$ is a ring.
\end{definition}

\begin{note}
    In practice it is the cases of $K = \Z$ (integers) and $K = \Z_2$ (integers mod 2) that are most often considered. The case $K = \Z_2$, as we will shortly see, is by far the more subtle and is the case considered in this paper.
\end{note}

\begin{definition}
\label{defn112}
    If $M_c$ is a manifold with a fixed triangulation\footnote{Throughout, triangulation refers to a PL triangulation.} we can similarly define the $(p,q)$-combinatorial systolic ratio of $M_c$, $K \mathdash (p,q) \mathdash \operatorname{SR}_c(M_c)$ as $\inf_{P,Q} \frac{\#_p(P) \cd \#_q(Q)}{\#_d(M)}$, where $\#_{p(q)}$ is the number of $p(q)$-simplexes comprising $P(Q)$, an essential $p(q)$-cycle$\in H_{p(q)}(M;K)$, where $[P] \cd [Q] \neq 0 \in H_0(M;K) \cong K$.
    When counting simplexes in a cycle, we count $|{\rm multiplicity}|$, so $\#(\sum_i a_i \sigma_i)=\sum_i |a_i|$.
\end{definition}

In both Riemannian and combinatorial contexts a manifold $M$ has \emph{$K \mathdash (p,q)$-systolic freedom} if $M$ admits a sequence $\{i\}$ of Riemannian metrics (triangulation) so that $K \mathdash (p,q) \mathdash \operatorname{SR}_R(M_i) \ra \infty$ ($K \mathdash (p,q) \mathdash \operatorname{SR}_c(M_i) \ra \infty$). One can also say that an arbitrary sequence $M_i$ has systolic freedom if these quantities approach 0, without requiring the members of the sequence $\{M_i\}$ to be all diffeomorphic (combinatorially equivalent).

It is interesting to go further and attempt to measure the \emph{amount} of systolic freedom in a sequence $\{M_i\}$ (of any type Riemannian or combinatorial, constant or variable topology). This may be directly implemented in the combinatorial setting where one may say that there is \emph{poly-log} or even \emph{power-law} systolic freedom if in the systolic ratios $\operatorname{SR}(M_i) = \frac{\text{numerator}_i}{\text{denominator}_i} = \frac{n_i}{d_i}$, $n_i = \Omega(\log(d_i)^\alpha d_i)$, or $n_i = \Omega(d_i^{(1+\alpha)})$ (respectively) for some $\alpha>0$, where the $\Omega$-notation means for sufficiently large $i$ ``greater than a positive constant times...'', while we later use $O(\cd)$ to denote ``at most a positive constant times..." and
$\Theta(\cd)$ to denote ``both $\Omega(\cd)$ and $O(\cd)$, i.e., at least a positive constant times and at most a positive constant times..."

There is an obvious difficulty in similarly attempting to quantify systolic freedom in the Riemannian case, as merely dilating all metric tensors $\{M_i\}$ by factors $\{l^2_i\}$ replaces $\{\frac{n_i}{d_i}\}$ with $\{\frac{l_i^d n_i}{l_i^d d_i}\}$, diminishing the power (or polylog) of divergence if $\{l_i\}$ diverges and increasing it if $\{l_i\}$ approaches zero. To avoid this problem, let us from here forward assume all sequences of Riemannian metrics are \emph{appropriately scaled}.

\begin{definition}
    A sequence of Riemannian metrics is appropriately scaled if:
    \begin{enumerate}
        \item All sectional curvatures lie in the range $[-O(1),O(1)]$ and
        \item inj.\ rad.$(M_i) = \Omega(1)$.
    \end{enumerate}
\end{definition}

\begin{thm}\label{thm:SRconstant}
    Fixing the dimension $d$, to any appropriately scaled sequence of closed Riemannian manifolds $\{M_i\}$ there corresponds a sequence of concretely Whitehead triangulated\footnote{The adjective ``Whitehead'' [W40] refers to his condition that all $k$-simplexes of the triangulation are uniformly biLipschitz equivalent to a Euclidean ball (of some radius) and that there is a uniform bound on the number of simplexes that any given simplex meets.} P.L.\ manifolds $\{M_{i,\epsilon}\}$ so that the three quantities in the definition of SR are all within constants:
    \[
         \frac{\inf \operatorname{area}_p(P_i)}{\#_p(P_i)} = \Theta(1),\  \frac{\inf \operatorname{area}_q(Q_i)}{\#_q(Q_i)} = \Theta(1), \text{ and } \frac{\operatorname{vol}_d(M_i)}{\#_d(M_i)} = \Theta(1)
    \]

    The constant does depend on the dimension $d$. Note that the middle condition is actually equivalent to the first since $0 \leq p,q \leq d$ are not further specified.
\end{thm}

This theorem allows us to work in P.L.\ or smooth categories according to convenience. Its proof will be given at the end of this section.

To understand $\Z$-systolic freedom, let us recall a basic example and theorem from \cite{gromov96}; also see D of \cite{gromov2007metric} and \cite{babenko1998systolic}.
For $i$ a positive integer, let $M_i=S_i^3 \times \mathbb{R}/(\theta,t)=(\sqrt{i} \circ \theta,t+1)$, where
$S_i^3$ is the $3$-sphere of radius $i$, and the identification matches a point with its $\sqrt{i}$-rotation
along Hopf fiber displaced one unit in the real coordinate.  This family of metrics on $S^3 \times S^1$ is already appropriately scaled.  A simple {\it calibration} argument shows that the $3$-systole is represented by
a fiber $S_i^3 \times t$, $t\in [0,1]$, but the shortest essential $1$-cycle has length $=\Theta(i^{1/2})$.
Notice that with a $1/2$-power $\sqrt{i}$ defining the amount of rotation, minimal $1$-cycles representing
$k[{\rm generator}]$ have minimal length fluctuating between $\Theta(i^{1/2})$ and $\Theta(i)$ as $k$ grows.
Thus we see a kind of $\mathbb{Z}$-systolic freedom:
$$\mathbb{Z}-(3,1)-{\rm SR} = \frac{\Theta(i^3) \cdot \Theta(i^{1/2})}{\Theta(i^3)},
$$
diverging with $i$.

Recalling Definition \ref{defn112}, this example represents power law $\Z$-systolic freedom, with power $\alpha = \frac{1}{6}$. But with $\Z_2$-coefficients, there is \emph{no} systolic freedom in the family. The idea for making low 3-volume $\Z_2$-cycles is to cut out large antipodally located pairs of patches on $S^3$ and heal the resulting boundaries by adding a length one tube $S^2 \times I$ joining the boundary of a patch to the boundary of its antipodal partner. This construction can reduce the area of the essential $P \in H_3(M_i; \Z_2)$ to $O(i^2)$. The first examples of sequences of metrics (the underlying manifold varied) with $\Z_2$-systolic freedom were given in \cite{flm02} and \cite{freedman99}; these examples exhibited poly-log freedom with $\alpha = \frac{1}{2}$.

The example above contrasts a more general absence of systolic freedom in a rational setting, where
``economies of scale" destroy many examples such as $\{M_i\}$ above.

For $[\alpha]\neq 0\in H_k(M;\mathbb{Z})/{\rm torsion}$, define the stable area
\be
{\rm starea}_k(\alpha)=\inf_j \frac{1}{j} \Bigl(\inf \limits_{\beta} {\rm area}_k(\beta)\Bigr),
\ee
where $j=1,2,3,\ldots$ and the inner infimum is over (rectifiable) cycles $\beta$ representing $j[\alpha]$.
Then define the stable systole {\it stsys}:
\be
{\rm stsys}_k=\inf\limits_{[\alpha] \neq 0} {\rm starea}_k(\alpha).
\ee

\begin{thm}
{\bf Gromov}: There is a constant $c(n)$ such that for any Riemannian metric on $M$ diffeomorphic to $S^p \times S^q$,
\be
{\rm vol}(M)={\rm stsys}_n(M)\geq c(n) \cdot {\rm stsys}_p(M) \cdot {\rm stsys}_q(M).
\ee
\end{thm}

With this general background in place, we embark on a study of $\mathbb{Z}_2$-systolic freedom, based on
$\mathbb{Z}_2$ quantum codes.  Integer (but not rational) coefficients play an important intermediary
role as we need their greater descriptive power to
make manifold constructions.
Here is our main theorem:
\begin{thm}
\label{powerZ2first}
    There is a sequence $\{M_i\}$ of closed, simply connected, stably framed Riemannian $11$-manifolds exhibiting power-law $\Z_2$-(4,7)-systolic freedom. The power is $1+\alpha$, $\alpha = 1$ up to polylogarithmic factors.
\end{thm}

We will now describe a general procedure, clarifying work in \cite{bh}, for turning a class of quantum codes, called CSS after the authors of [CSS], into manifolds. When this procedure is applied to certain recently developed CSS codes \cite{HHO}, the result will be, up to polylogarithmic factors, a sequence of Riemannian manifolds $\{M_i\}$ exhibiting power law $\Z_2$-systolic freedom with $\alpha = \frac{1}{4}$, given the distance of codes found in that work\footnote{The work constructs codes on $n$ qubits with distances (up to polylogarithmic factors) $n^{3/5}$, where the distance is the minimum weight of a nontrivial homology or cohomology representative.  This would gives only $\alpha=\frac{1}{5}$; however, these codes result from ``distance balancing" codes with distances $d_X=n^{1/2},d_Z=n^{3/4}$, up to polylogarithmic factors.  Thus, before balancing, one gets $\alpha=1/4$ up to logarithmic factors.}. This is the first example of power law $\Z_2$-systolic freedom. For clarity and simplicity we construct the sequence $\{M_i\}$ to be closed, simply connected, 11-manifolds. It is easy to push up the dimension. We will discuss whether the dimension can be lowered near the end of this section.

We will also apply this procedure to hypergraph product codes\cite{tillich2013quantum} later in the paper.

The quantum code should be a low-density parity check (LDPC) code, as defined below; this is needed to control the local geometry of the manifold we build.
The construction will require a particular further assumption on the class of codes, requiring that the code is what we term {\it sparsely liftable}, as defined below in definition \ref{liftabledef}.

Remark added: after the first version of the present paper appeared, the distance record for LDPC quantum codes was increased to give codes on $n$ qubits
with almost linear distance $d_X = d_Z =\Omega( n/\log(n))$.  See \cite{LD}; see also \cite{BE} for a derandomization of \cite{HHO} and a unified viewpoint relating the codes of \cite{LD,HHO}.  Applying our procedure to those codes gives power law $\Z_2$ systolic freedom for $\alpha=1$ up to logarithmic factors.

\subsection{CSS codes}
Recall a CSS code $\mathcal{C}$ [CSS] is given on $q$-qubits $Q_1, \dots, Q_q$, with $X$-stabilizers\footnote{We use the term ``stabilizer" as shorthand for ``stabilizer generator" throughout.} and $Z$-stabilizers $\{X_{i_1} \otimes \cdots \otimes X_{i_j}\}$, $1 \leq i_1 < \cdots < i_j \leq q$; and $\{Z_{k_1} \otimes \cdots \otimes Z_{k_l}\}$, $1 \leq k_1 < \cdots k_l \leq q$. It is required that all stabilizers commute (which is nontrivial only between the $X$- and $Z$-stabilizers).

As usual $X_i$ and $Z_i$ denote the operators

\begin{center}
\begin{tikzpicture}
    \node at (0,0) {$X_i =\quad \begin{vmatrix} 0 & 1 \\ 1 & 0 \end{vmatrix}$ on $Q_i$, and $Z_i = \quad \begin{vmatrix} 1 & 0 \\ 0 & -1 \end{vmatrix}$ on $Q_i$};
    \node at (1.3,-0.25) {\scriptsize{$|1\rangle$}};
    \node at (1.3,0.25) {\scriptsize{$|0\rangle$}};
    \node at (1.9,0.7) {\scriptsize{$|0\rangle$}};
    \node at (2.5,0.7) {\scriptsize{$|1\rangle$}};
    \node at (-2.9,-0.25) {\scriptsize{$|1\rangle$}};
    \node at (-2.9,0.25) {\scriptsize{$|0\rangle$}};
    \node at (-2.4,0.7) {\scriptsize{$|0\rangle$}};
    \node at (-1.9,0.7) {\scriptsize{$|1\rangle$}};
\end{tikzpicture}
\end{center}

It is well-known in quantum coding theory that the condition that the $X$- and $Z$- stabilizers commute, i.e.\ that given any $X$- and any $Z$-stabilizer the number of common indices is even, means that a \emph{chain complex} with three different degrees may be defined obeying $\de^\pr \circ \de = 0$:
\begin{equation}\label{eq:chaincomplex}
    \mathcal{C} = \Z_2^z \xrightarrow{\de} \Z_2^q \xrightarrow{\de^\pr} \Z_2^x
\end{equation}
where $z = \#(Z\text{-stabilizers})$, $q = \#(\text{raw qubits})$, and $x = \#(X\text{-stabilizers})$. 
The three different degrees are given preferred bases, correspondingly respectively to $Z$-stabilizers, qubits, and $X$-stabilizers.
The maps $\de^\pr$
and $\de$ are defined 
by their matrix elements $(\de^\pr)_{ij},(\de)_{ij}$ in this basis by:
\begin{equation}
\begin{split}
    &  (\de^\pr)_{ij} =
    \begin{cases}
        0 & \text{if } X_j \text{ does not occur in the }i\text{th } X\text{-stabilizer} \\
        1 & \text{if } X_j \text{ occurs in the }i\text{th } X\text{-stabilizer}
    \end{cases} \\
    & \text{for } 1 \leq i \leq x \text{ and } 1 \leq j \leq q \text{, and} \\
    &  (\de)_{ij} =
    \begin{cases}
        0 & \text{if } Z_i \text{ does not occur in the }j\text{th } Z\text{-stabilizer} \\
        1 & \text{if } Z_i \text{ occurs in the }j\text{th } Z\text{-stabilizer} \\
    \end{cases} \\
    & \text{for } 1 \leq i \leq q \text{ and } 1 \leq j \leq z
\end{split}
\end{equation}
See \cite{bh} for a review and for a ``dictionary" mapping between these languages: for example, so-called ``$Z$-logical operators" in quantum coding theory which act nontrivially on encoded qubits correspond to homology representatives while ``$X$-logical operators" correspond to cohomology representatives.

Let us define $f(i)$ to be the number of $X$-stabilizers containing $X_i$, i.e., the number of ones in the $i$-th column of $\de^\pr$.

The condition that $\mathcal{C}$ is a low-density parity check code (LDPC) says that there are only $O(1)$ nonzero entries in every row and every column of $\de$ and in every row and every column of $\de^\pr$.

Our construction will build a manifold starting with a chain complex over $\Z$, rather than $\Z_2$.  So, we define lifts as follows.
We define a ``lift" of a matrix $M$ of entries in $\Z_2$ to be a matrix $\tilde M$ over $\mZ$ with the same number of rows and columns such that each entry of $\tilde M$ is equal, mod $2$, to the corresponding entry of $M$.

Of course every matrix has some lift that preserves sparsity: imply replace each $0$ or $1$ in $\F_2$ with $0$ or $1$ correspondingly in $\mZ$.  We call this the ``naive lift".

\begin{definition}
Given a chain complex over $\Z_2$, defined by boundary operators $\de_j:{\mathcal A}_j \rightarrow {\mathcal A}_{j-1}$ for some sequence of ${\mathcal A}_j$ (in the above example, $3$ successive values of $j$ corresponding to $z,q,x$),
where the boundary operators are regarded as finite matrices with entries $0,1$,
a {\it lift} of that chain complex is a chain complex over $\Z$ defined by a sequence of boundary operators $\tilde \partial_j$,
such for all $j$, $\tilde \partial_j$ is a lift of $\partial_j$.
\end{definition}

\begin{definition}
\label{liftabledef}
A chain complex over $\Z_2$ is {\it liftable} if it admits a lift.  Such a chain complex is {\it sparsely liftable} if it admits a
lift such that for every row and every column of every boundary operator of the lifted complex, the sum of absolute values of entries of that row or column is $O(1)$.
\end{definition}

{\bf Notation:} let us use $\tilde f(i)$ to denote the sum of absolute values of the $i$-th column of $\de^\pr$.  Sparseness of the lift implies $\tilde f(i)=O(1)$.

As mentioned, every matrix (such as $\de$) has some lift, such as the naive lift above.  However, it is not obvious that a pair of boundary operators admit a lift so that $\tilde \de^\pr \tilde \de=0$ over integers.  In fact, we will show in section \ref{liftingsection} that such lifts do always exist (all complexes are liftable), but we conjecture that for many LDPC codes, no {\it sparse} lift exists.

In all that follows in the rest of this section, the chain complexes used to define the manifold are obtained from
a lift (or, better, a sparse lift).  We will omit the use of the tilde (i.e., $\tilde f(i)$, $\tilde \de$, or $\tilde \de^\pr$) and simply use $\de$ or $\de^\pr$ to implicitly refer to lifted boundary operators.

As better quantum codes are found, stronger systolic inequalities will follow.  So let us state a general
template of which Theorem \ref{powerZ2first} is a special case.

\begin{thm}
\label{generaltemplate}
Given a sparse $\mathbb{Z}_2$-chain complex $E_2 \rightarrow E_1 \rightarrow E_0$,
our input LDPC quantum code, and a sparse lift to an integral chain complex
$\tilde E_2 \rightarrow \tilde E_1 \rightarrow \tilde E_0$, there is a simply connected, closed,
orientable $11$-manifold which may be taken to be PL triangulated, smooth, or both (using a Whitehead
compatible triangulation) so that:
\begin{itemize}
\item[(1):] The triangulation has bounded local degree (or in Riemannian terms, the metric is appropriately scaled).

\item[(2):] The number of simplices (or Riemannian volume) is $\Theta(\sum_i {\rm dim}(E_i))$ up to polylogs.

\item[(3):] The mod $2$ $4$-systole of $M$ is $\Theta({\rm sys}_1(E))$, in either the PL or Riemannian sense.

\item[(4):] The mod $2$ $7$-systole of $M$ is $\Theta({\rm cosys}_1(E))$, in either the PL or Riemannian sense.

\item[(5):] In fact, $M$ can be given a handle decomposition so that the portion of the cellular chain
complex:
$${\rm span}_{\mathbb{Z}}(5-{\rm handles}) \xrightarrow{\rm attaching}
{\rm span}_{\mathbb{Z}}(4-{\rm handles}) \xrightarrow{\rm attaching}
{\rm span}_{\mathbb{Z}}(3-{\rm handles})$$
is isomorphic to
$\tilde E_2 \rightarrow \tilde E_1 \rightarrow \tilde E_0$.
\end{itemize}

Furthermore if it is desired to map the code to a manifold of higher dimension, $11+d$, then if we can write $d\geq 0$ as $d=p+q$, $p,q\geq 0$, then in all statements above we may replace:
$$11\rightarrow 11+d, \,
3\rightarrow 3+p, \,
4\rightarrow 4+p, \,
5 \rightarrow 5+p,\,
7\rightarrow 7+q.$$
\end{thm}

\subsection{Bounded Local Geometry}
Our construction will give a handle decomposition of some manifold.
We wish our construction to have
the following desirable feature $(*)$: if the original code is LDPC, then the resulting decomposition has a {\it bounded local geometry}, meaning that only $O(1)$ handles are attached to any given handle.
If a manifold does not have bounded local geometry, we say that it is {\it congested}.

The requirement that the handle decomposition have bounded local geometry, so that the manifold admits a bounded Riemannian metric, explains some of the unusual features of our construction below.  In particular, the usual start to a handle decomposition of a connected manifold --- a single $0$-handle --- will lead to congestion as every other handle would make contact with it.

\subsection{Naive Attempt}
The reader might be surprised that our construction builds an eleven dimensional manifold, rather than some lower dimensional object.  Indeed, in \cite{bh}, a procedure to build a seven manifold was sketched.  Let us review this procedure, which we describe as a ``naive" version of the present construction.  In this naive procedure, we will give a cell decomposition, rather than a handle decomposition, and then thicken the cell decomposition to obtain a manifold.  We define bounded local geometry for a cell complex in the obvious way, meaning that every cell attaches to $O(1)$ cells, and define $(*)$ similarly for a cell decomposition.

We will explain some issues that this naive procedure encounters, and explain that these issues become more severe
when one attempts to ensure feature $(*)$. To resolve these issues, while also having this  feature $(*)$, we will need to increase the dimension of the cells used in the construction. 

The basic idea of both the naive construction, and the construction here, is to fix some $j$ and
identify $(j-1)$-cells  with $X$-stabilizers,
identify $j$-cells with qubits,  and identify $(j+1)$-cells with $Z$-stabilizers, and use the boundary operators to give rules for attaching the cells to each other.

The smallest $j$ that one might imagine using is $j=1$.  In this case, we immediately encounter a problem:
a $1$-cell can have at most two $0$-cells in its boundary, which limits the possible maps $\de^\pr$ to matrices where each column has at most two nonzero entries.

In this regard, let us remark on the term ``cell".  In this section of the paper, a ``cell" will refer to a particular topological space.  A cell should be a ball, however we will sometimes put quotes around the word, i.e. ``cell", to refer to some object that is not a ball.  However, in section \ref{liftingsection}, we will use the term cell more abstractly, to refer to a particular basis element of a vector space.

So, one might then turn to $j=2$, which is what was suggested in \cite{bh}.  
Then, we have one $1$-cell for each $X$-stabilizer.  There it was suggested to then use the boundary operators to give a rule
to attach $2$-cells for qubits and $3$-cells for $X$-stabilizers, giving a $3$-complex.  Then, this complex can be embedded in $7$-dimensions, thickened to give a manifold with boundary, and attached to a copy of itself to give a closed manifold.

Let us see what happens.
There are a few possible choices already, even before attaching $2$-cells.  We might have a single $0$-cell, and attach all the $1$-cells to that $0$-cell, giving a bouquet of circles with one $S^1$ for each $X$-stabilizer.  Unfortunately, this will not give feature $(*)$: the $0$-cell has {\it all} of the $S^1$ attached to it, rather than having only $O(1)$ cells attached to it.  Nevertheless, let us continue.  There is no difficulty attaching $2$-cells.  Each $2$-cell attaches to some combination of the $1$-cells, which can be itself chosen to be spherical (i.e., $S^1$) so a $2$-ball can attach to that boundary.  However, when attaching $3$-cells, some difficulty may arise.  Since $\de^\pr \de=0$, the ``$3$-cell" should attach to some closed $2$-chain.  However, we do not in general know that this $2$-chain is spherical.  Thus, the ``$3$-cell" might not be a ball.  This can lead to some spurious homology appearing: the manifold might have extra second homology compared to the homology of the qubits of the chain complex.

Rather than dwelling on this issue (which we will resolve by going to larger $j$), let us also consider the issue that this construction does not give $(*)$.  A different approach then is to 
start with ``free-floating" $S^1$, with one $S^1$ for each $X$-stabilizer.  These $S^1$ are not attached to a fixed $0$-cell, so that before attaching $2$-cells the $S^1$ are disconnected from each other.  Now  attach the ``$2$-cells".  Now the issue of spurious homology occurs already at this step.  To give an example, suppose we have two qubits, labelled $1,2$, and two $X$-stabilizers, both of the form $X_1 X_2$.  There are no $Z$-stabilizers.  In this case, the matrix $\de^\pr$ is
$$\begin{pmatrix} 1 & -1 \\ 1 & -1\end{pmatrix},$$
where we have have chosen some particular lift.
Then, we have two $S^1$, and each ``$2$-cell" attaches to both $S^1$.  Unfortunately, the ``$2$-cells" are not balls but rather annuli.  With the choice of signs, this gives a torus $T^2$, which has Betti number $b_1=2$.  This is some spurious homology: the first Betti number should correspond to the $X$-homology of the chain complex, i.e., the redundancy of the $X$-stabilizers, and so it should instead be $1$.
This problem is not so serious yet, but after further steps of the construction it can also lead to spurious second homology, which may destroy the systolic freedom\footnote{Yet another attempt would be to start with a bounded degree tree graph, with one leaf per $X$-stabilizer and attach $S^1$ to the leaves; this will typically not give feature $(*)$ when attaching $2$-cells since they will need to ``stretch across the tree" and there may be many cells attached to a given edge.}.

Our approach below will mimic the ``free floating" construction, starting with disconnected $S^{j-1}$ for $j \geq 4$, one such $S^{j-1}$ for each $X$-stabilizer.  By taking the larger $j$, we will be able to avoid all these issues and get $(*)$.  We will in some cases use ``cells" rather than cells, but we will be able to add extra lower dimensional cells to turn the higher dimensional ``cells" into genuine cells.

Also, rather than working with thickenings of a cell complex, we directly give a handle construction of a manifold which allows us to more easily identify this cell decomposition.

\subsection{Building a $2j+3$-dimensional $M$ from $\mathcal{C}$, $j \geq 4$}
For specificity of visualization (do we visualize 11D manifolds?) in this subsection we will just set $j=4$ and build a $2 \cd 4 + 3 = 11$-manifold. So in the construction $3 = j-1$, $4 = j$, $5 = j+1$, larger integers, 6 and higher, are interpreted dually so, for example, 7 should be thought of as $2j+3-4$, indeed 4 and 7 cells are in dual dimensions. It is hoped that this (unnecessary) specificity will save the reader some mental arithmetic.

The input to the construction is a sparse integral lifting $\tilde C$ of a CSS code, but following our convention we will suppress the tilde.  $\Z_2$ information alone is not adequate to specify attaching maps nor to ensure the orientability of the $4$-cycle we build.

An 11D $k$-handle is a pair $(D^k \times D^{11-k}, \de D^k \times D^{11-k})$

\begin{figure}[ht]
    \centering
    \begin{tikzpicture}
        \draw (-3,-1) rectangle (3,1);
        \draw[very thick] (-3,-1) -- (-3,1);
        \draw[very thick] (3,-1) -- (3,1);
        \draw (0,-1) -- (0,1);
        \draw (-3,0) -- (3,0);
        \node at (1,0.5) {$\leftarrow$co-core};
        \node at (1.4,-0.5) {core};
        \draw[->] (0.95,-0.5) to [out=155,in=-80] (0.8,-0.2);

        \draw [decorate,decoration={brace,amplitude=5pt}](3.4,1) -- (3.4,-1) node [black,midway,xshift=1.8em,yshift=2pt] {$D^{11-k}$};
        \node at (6.5,0.25) {$0 \times D^{11-k}$ is called};
        \node at (6.5,-0.25) {the \emph{co-core}};
        \node at (6.5,-0.75) {$D^k \times 0$ is the \emph{core}};

        \draw [decorate,decoration={brace,amplitude=5pt}](3,-1.2) -- (-3,-1.2) node [black,midway,yshift=-12pt] {$D^k$};

        \node at (-4.4,-1.9) {$\de D^k \times D^{11-k}$};
        \node at (-4.4,-2.4) {called \emph{attaching region}};
        \draw[->] (-4.2,-1.7) -- (-3.2,-0.8);
        \draw [->] (-3.7,-1.25) to [out=90,in=190] (-2.4,1.5) to [out=10,in=170] (2.4,1.5) to [out=-10,in=55] (3.2,0.8);
        \node at (-8,0) {$\vspace{0.5em}$};
    \end{tikzpicture}
    \caption{}
    \label{fig1}
\end{figure}

Handlebodies are unions of a succession of handles where each handle is attached (or ``glued") along its attaching region to the previous union of handles.
Handlebodies are built by starting with 0-handles (all handles for some time now will be 11-dimensional) which are just 11-cells, $(0 \times D^{11},\varnothing)$ with \emph{empty} attaching region. That is why one starts with 0-handles: you don't need anything to attach them to! Then the usual procedure is to step by step attach handles of increasing indices, always to the boundary of what has been previously constructed until one terminates by attaching 11-handles $(D^{11} \times 0, S^{10} \times 0)$, which have empty \emph{co-attaching region} $D^k \times \de D^{11-k}$ to close the 11-manifold.

Handlebodies are a good way to build and understand smooth manifolds. To keep things smooth the attaching is always by diffeomorphism from attaching region to its image. The technology also works well in the P.L.\ category.

In our exposition we will build a 5-handlebody $H$, meaning a union of 11D handles of indices 0, 1, 2, 3, 4, and 5. Then $M = DH$ is obtained as the double of $H$: two copies of $H$ (one with reversed orientation) glued together along $\de H$ by $\id_{\de H}$. By duality, in the upside down copy $k$-handles become $(11-k)$-handles, so all indices $0, 1, \dots, 11$ occur in our description of $M$. Two aspects will be a little unusual: We will not attach the handles strictly in index order, we will do indices 0 and 3 before later getting to 1 and 2. Also we will talk about attaching ``4-handles'' and ``5-handles.'' The scare quotes are a warning to the reader that these are not handles but themselves small handlebodies containing handles of lower indices, which later we will separate into their constituent handles. However, when first presented, we will not see their internal handle structures, we will see them as $(N^4 \times D^7, \de N^4 \times D^7)$ or $(N^5 \times D^6, \de N^5 \times D^6)$ for manifolds with boundary $N^4$ or $N^5$. Later choosing a handle structure on $N^4$ or $N^5$ (and crossing this with $D^7$ or $D^6$) will yield the ``small handlebodies'' mentioned above.

Finally, let us explain an important general principle that ``the
construction is local" that enables us to preserve bounded local geometry through all
steps of this construction.
Our construction goes through several steps, as we attach additional ``handles", decompose ``handles" into handles, and (after forming the double) re-order the attachment of handles by index. 
After every step, the local geometry will be bounded.
On steps where we attach a ``4-handle" or ``$5$-handle", which will correspond to a basis vector of $C$ (i.e., a qubit or $Z$-stabilizer), the
rule for the attaching map will depend only on nonzero matrix elements of $\partial'$ or $\partial$ in the given column, as well as on the  (previously constructed) local geometry of the ``handles" corresponding to basis vectors in the boundary of the given basis vector. 
If the lifted $\partial'$ and $\partial$ are sparse, then there are only $O(1)$ possible choices of local geometry then that one need consider, and so we may choose the local geometry for the attached ``handle" to be bounded.
On other steps, the local geometry after the step is determined by the local geometry before the step and hence again must remain bounded.

So let us begin with $x$ 0-handles, one for every $X$-stabilizer. Now attach to each of these a single 3-handle. The result so far is $(0,3)$-handlebody diffeomorphic to $\coprod_{j=1}^x (S^3 \times D^8)_j$. This space, call it $X$, represents in grade 3 the rightmost chain group on line \ref{eq:chaincomplex}.

\begin{figure}[ht]
    \centering
    \begin{tikzpicture}
        \draw (0,0) circle (0.8);
        \draw (0,0) circle (1.3);
        \draw (-3.5,0) circle (0.8);
        \draw (-3.5,0) circle (1.3);
        \draw (4.5,0) circle (0.8);
        \draw (4.5,0) circle (1.3);
        \node at (2.25,0) {$\dots$};
        \draw [decorate,decoration={brace,amplitude=10pt}](5.6,-1.3) -- (-4.7,-1.3) node [black,midway,yshift=-18pt] {$X = x$ copies of $S^3 \times D^8$};

        \draw (0,0.8) -- ++ (90:0.5);
        \draw (0,-0.8) -- ++ (270:0.5);
        \draw (0.8,0) -- ++ (0:0.5);
        \draw (-0.8,0) -- ++ (180:0.5);
        \draw (0.566,0.566) -- ++ (45:0.5);
        \draw (0.566,-0.566) -- ++ (-45:0.5);
        \draw (-0.566,0.566) -- ++ (135:0.5);
        \draw (-0.566,-0.566) -- ++ (225:0.5);

        \draw (-3.5,0.8) -- ++ (90:0.5);
        \draw (-3.5,-0.8) -- ++ (270:0.5);
        \draw (-2.7,0) -- ++ (0:0.5);
        \draw (-4.3,0) -- ++ (180:0.5);
        \draw (0.566-3.5,0.566) -- ++ (45:0.5);
        \draw (0.566-3.5,-0.566) -- ++ (-45:0.5);
        \draw (-0.566-3.5,0.566) -- ++ (135:0.5);
        \draw (-0.566-3.5,-0.566) -- ++ (225:0.5);

        \draw (4.5,0.8) -- ++ (90:0.5);
        \draw (4.5,-0.8) -- ++ (270:0.5);
        \draw (4.5+0.8,0) -- ++ (0:0.5);
        \draw (-0.8+4.5,0) -- ++ (180:0.5);
        \draw (0.566+4.5,0.566) -- ++ (45:0.5);
        \draw (0.566+4.5,-0.566) -- ++ (-45:0.5);
        \draw (-0.566+4.5,0.566) -- ++ (135:0.5);
        \draw (-0.566+4.5,-0.566) -- ++ (225:0.5);
    \end{tikzpicture}
    \caption{}
\end{figure}

Now for each qubit $Q_i$, $1 \leq i \leq q$, prepare a ``4-handle'' to attach to $X$ according to $\de^\pr$. Specifically, for the $i$th qubit $Q_i$, $1 \leq i \leq q$, let $N_i = S^4 \setminus \operatorname{int}(f(i)\text{ disjoint 4-balls})$. Form ``4-handles'' $\{(N_i \times D^7,\de N_i \times D^7)\}$, $1 \leq i \leq q$, and attach these via disjoint embedding of the attaching regions into $\de(X) \cong \coprod_{j=1}^x S^3 \times S^7$. Each embedding $(S^3 \times D^7)_{i,m} \hookrightarrow (S^3 \times S^7)$, $1 \leq i \leq q$, $1 \leq m \leq f(i)$, should have homological degree $= \pm 1$ in $H_3(S^3 \times S^7,\Z)\cong \Z$, so that
taken together these degrees specify the lifted boundary map $\de$.
For example, if the lifted boundary map sends $Q_i$ to $2X_1 +3X_2$, $N_i$ would be $S^4$ with five balls removed, and two of the resulting boundary spheres would glue to $X_1$ and three would glue to $X_2$.

In case the reader is wondering, in these dimensions there is no knotting or linking of 3-spheres so the smoothing isotopy class of the core 3-spheres is determined by these $\pm 1$ choices. There is an additional $\pi_3(\operatorname{SO}(7)) \cong \Z$ normal framing choice for each embedded $S^3$. This is specified by choosing specific tangential framing for all the 11-manifold pieces and requiring that all ``handle'' attachments are compatible with these framings. The reason for specifying framings is twofold.  First, it makes
the manifold an explicit function of the lifted quantum code.  Second, our framing choices yield an $M^{11}$
with stably framed tangent bundle.  This facilitates the analysis of diffeomorphism type carried out in section \ref{difftypesec}.

Specifically, write $S^3 \times D^8 \cong S^3 \times I \times D^7$, frame the tangent bundle $\tau(S^3 \times I)$ by regarding it as a collar on $\de D^4$ and restricting the standard coordinate frame on $D^4$ to $S^3 \times I$. Then extend this to an 11-frame on $\tau(S^3 \times I \times D^7)$ by adding the standard coordinate frame on $D^7$ as the last seven of the 11 frame vectors. Similarly, write $N_i \times D^7 \cong N_i \times I \times D^6$.  
Regard each component $\partial N_i \times I$ as a collar on the boundary of one of the $f(i)$ deleted $4$-balls
$D^4$
by starting with the embedding $\de N_i \subset S^4$, and extending the interval coordinate radially. Again restrict the coordinate frame on $D^6$ to $N \times I$ and extend with the coordinate frame on $D^6$. We require that when ``handle'' attachments are \emph{sunken into the bulk} the framing must match (up to homotopy). Smooth manifolds are defined through overlapping charts, a similarly convenient way to define smooth structures---and in the present case---also tangential framings is to extend handle attachment to a gluing of overlapping bulks as in Figure \ref{fig:sunkenbulk}.

\begin{figure}[ht]
    \centering
    \begin{tikzpicture}
        \draw (2.5,0) ellipse (1.5 and 1);
        \draw (-2.5,0) ellipse (1.5 and 1);
        \draw (-3.3,0.83) to[out=100,in=180] (-2.5,2.3) to[out=0,in=80] (-1.7,0.83);
        \draw (-2.92,0.95) to[out=100,in=180] (-2.5,1.9) to[out=0,in=80] (-2.08,0.95);
        
        \draw (5-3.3,0.53) to[out=100,in=180] (2.5,2) to[out=0,in=80] (5-1.7,0.53);
        \draw (5-2.92,0.65) to[out=100,in=180] (2.5,1.6) to[out=0,in=80] (5-2.08,0.65);
        \draw (3.3,0.53) -- (2.92,0.65);
        \draw (1.7,0.53) -- (2.08,0.65);
        
        \node at (-0.4,2.2) {conventional};
        \node at (-0.4,1.85) {attachment};
        \draw[->] (-0.4,1.6) -- (-1.55,1);
        \node at (5,2.2) {attachment extended};
        \node at (5.2,1.85) {to bulk};
        \draw[->] (5.2,1.6) -- (3.5,0.9);
        
        \node at (-6.5,0) {$\hspace{0.5em}$};
    \end{tikzpicture}
    \caption{}\label{fig:sunkenbulk}
\end{figure}

Let us name the disjoint union $\coprod_{i=1}^x S^3 \times D^8$, $X$ and use the name $QX$ for the 11-manifold made by attaching the $q$ ``4-handles'' as just specified.

Next we define the ``5-handles'' $Z$ and their attachment to $QX$ to form $ZQX$.

Let $Z^k$ be one of the $Z$-stabilizers, $1 \leq k \leq z$. The indices participating in $Z^k$, or equivalently the image of $\de$ applied to $Z^k$, defines a 4-cycle in the middle chain group ($\ref{eq:chaincomplex}$). On the manifold level the $\{Q_i\}$ where $i$ appears in $Z^k$ define an element of $H_4(QX;\Z)$. There are several ways to represent this class by disjoint embedding of manifolds of the form $Y_i \coloneqq (\#_{\text{copies}} S^1 \times S^3) \times D^6 \subset \de(QX)$. A theorem of Euler notes that even valence graphs may be covered by disjoint cycles. Representing $Z^k$ by $\coprod_{i \in \text{Index}} Y_i$ is an application of that proof. The graph is the bipartite graph $H$ whose red vertices are the qubits $\{Q_i\}$, $i$ participating in $Z^k$, and whose black vertices are the $X$ stabilizers with indices also appearing in $\{Q_i\}$, edges dictated by $\de^\pr$. The exactness of (\ref{eq:chaincomplex}), now over the integers, implies that all the black indices are not only even valent, but see algebraically zero signed $\frac{1}{2}$-edges. Any pairing of $\frac{1}{2}$-edges of opposite signs at the black vertices provides a combinatorial model for resolving the homology class into an embedded disjoint union $Y = \coprod Y_i$.  That is, starting with a collection of punctured $S^4$'s (red), their boundary components are hooked up by pairs of $\frac{1}{2}$-edges at the black vertices.  The result is $Y$.

We can now define a ``5-handle'' $T^k$ modeled on $Z^k$ as follows: For each $Y_i$, build a graph $G_i$, $\operatorname{Betti}_1(G_i) = \operatorname{Betti}_1(Y_i) = \#(\text{copies}_i)$. Now set $G = \coprod_{i \in I} G_i$ and embed $G$ into $S^5$. The ``5-handle'' $T^k$ will be
\begin{equation}
    T^k = \Big[(S^5 \setminus \operatorname{int}(\mathcal{N}(G))) \times D^6, \de(S^5 \setminus \operatorname{int}(\mathcal{N}(G))) \times D^6 \Big]
\end{equation}
where $\mathcal{N}(G)$ is a closer regular neighborhood.

Observe that the attaching region of $T^k$ is bijective with $Y$ and that is how $T^k$ is to be attached.  Form $ZQX$ by attaching $\{T^k\}$ for all $1 \leq k \leq z$.

Now let us give $ZQX$ the structure of an actual handlebody. A typical ``4-handle'' $Q_i$ has a core $C_i$ consisting of $C_i \coloneqq S^4 \setminus f(i)$ punctures, that is, $S^4$ with the interior of several disjoint 4-balls removed. Relative to its boundary, $C_i$ has a 4D handle structure consisting of $f(i)-1$ 1-handles, to connect up the boundary components.  No more than $O(1)$ 1-handles should attach to any one boundary component in order to avoid congestion.
 Finally add a single 4-handle (the top cell). Crossing this handlebody structure with $D^7$ gives a true 11D handlebody structure on $Q_i$.

Next consider a typical ``5-handle'' $Z^k$, and its ``core'' $C^k$, $C^k \coloneqq (S^5 \setminus \operatorname{int}(\mathcal{N}(G)))$. Build a 5D handlebody structure on $C^k$ relative to its boundary by locating just enough 1-handles attached to the boundary to connect all the components (again with only $O(1)$ number attaching to any given component to avoid congestion). ``Locating'' these 1-handles in $C^k$ simply consists of choosing a disjoint embedded arcs in $C^k$ joining the boundary components and then thickening them to copies of $(D^1 \times D^4, \de D^1 \times D^4)$. At this point the upper boundary of the handlebody is a $\#_{n\text{-copies}} S^1 \times S^3$.

Just as we located 1-handles in $C^k$, we can locate precisely a sufficient quantity of 2-handles to kill a free basis for $\pi_1$(upper boundary), while retaining bounded geometry. 
Since each 
``5-handle" has some bounded local geometry, we only need $O(1)$ such 2-handles.
The ambient space being 5D, general position still makes maps of 2-disks generically embedded and disjoint. These disjoint disks are now thickened into 2-handles $(D^2 \times D^3, \de D^2 \times D^3)$. 
The framing data for the attachment of the ``5-handles" consists of choices of stable normal framings for the circles, the cores of the attaching regions of the 2-handles within the ``5-handles", and then a stable normal framing of $S^4$, the core of the attaching region of the actual 5-handle within the ``5-handle". The first choices, from $\pi_1$ of the stable orthogonal group, are made so as not to produce any second Stiefel-Whitney class in $\tau(ZQX)$, i.e. to ensure $ZQX$ is framed. The second choice actually does not exist as $\pi_4$ of the stable orthogonal group vanishes.
More directly, the original embedding $G \hookrightarrow S^5$ we used could easily have been extended to a $2$-complex $G^+$ to include the 1-handle cores and 2-handles cores in the first place:
\begin{center}
    \begin{tikzpicture}
        \node at (0,0) {$G \hookrightarrow S^5$};
        \node[rotate around={90:(-0.4,-0.5)}] at (0.4,-0.45) {$\hookleftarrow$};
        \node at (-0.4,-1) {$G^+$};
        \node[rotate around={45:(0.4,-0.2)}] at (0.2,-0.25) {$\hookrightarrow$};
    \end{tikzpicture}
\end{center}

After locating appropriate 1-handles and 2-handles in $C^k$, the handlebody structure of $C^k$ is completed by adding a single 5-handle, the top cell. Now, as before, cross with $D^6$ to obtain a true 11D handle structure on each ``5-handle.'' This completes the construction of $ZQX$.

{\bf Killing $\pi_1$:}
$\pi_1(ZQX)$ is freely generated by the 1-handles just described. We could ignore their fundamental group and just double $ZQX$ to produce our examples of power law $Z_2$-(4,7) systolic freedom but it is little extra trouble to kill this fundamental group and arrive at a simply connected family of examples.
We do this by attaching 2-handles whose boundary kills a basis for the fundamental group, forming $ZQX^+$.
We need to bound the volume of handles that are attached.  Previously, when considering a single ``5-handle" corresponding to a
$Z$-stabilizer, we appealed to the assumption of bounded local geometry so that only $O(1)$ disks needed to be attached.
Now, however, the cycles forming a basis of the fundamental group may be of length $\sim n$, where $n$ is the number of cells.
We appeal to the ``decongestion lemma" \ref{dcl}, which shows that on a sparse graph of $V$ vertices, one can find a weakly fundamental cycle basis so that each edge appears at most $O(\log(V)^2)$ times in the basis.
Each cycle can bound a disk of $2$-area proportional to the length of the cycle (and hence, after thickening, $11$-volume similarly bounded), so we can kill all cycles with at most
a logarithmic increase in volume.
As explained in appendix \ref{cyclebasis}, such a weakly fundamental cycle basis has the property that this disk attachment trivializes $\pi_1$; a more general cycle basis might not have this property and might only kill first homology.

For our purposes, it is important that this cycle basis has the stronger property that each edge appears only polylogarithmically many times in the basis, rather than just the property that the {\it total} number of edges used, counted with multiplicity, is bounded by a log times the number of edges.  Since each edge appears only polylogarithmically many times, we can keep polylogarithmically bounded geometry in each cell.  Indeed, we can even (at a further polylogarithmic increase in volume) change it so that every cell attaches to only $O(1)$ other cells.  Simply subdivide each cell into polylogarithmically different subcells in some direction
in the $11$-manifold (this leads to the polylogarithmic increase in volume) that we call the ``height" direction.  Attach each $2$-disk at some given ``height", so that only $O(1)$ disks attach to any subcell.  Indeed only a polylogarithmic number of subcells is required to do this, since for any cycle in the basis, there are only polylogarithmically many cycles later in the basis with which it shares an edge or vertex.  See for example lemma 3 of \cite{hastings2017weight}.  
A simpler argument (that gives a weaker bound, but still polylogarithmic on the needed number of subcells) is to regard this as a problem of hypergraph coloring.  Regard each $2$-disk that we attach as a vertex; regard the ``height" of that $2$-disk as a color.  Regard the cells (before subdividing into subcells) as hyperedges of a hypergraph; then the problem is to color that hypergraph so that each color appears at most $O(1)$ times in each, and standard coloring arguments show then that this can be done with polylogarithmically many colors.
We need to further refine the cellulation of the attached $2$-disks also so that each $2$-disk attaches to only $O(1)$ cells.

As remarked earlier, the actual handle body structure we have described on $ZQX^+$ is in an unconventional order; the handles are not attached lowest index first: our “4-handles" actually contain both 1- and 4-handles, our “5-handles" handles of index 1,2,and 5, and finally, at the last step we attached 2-handles to kill the fundamental group. This is unusual but not forbidden; they can, of course, be pushed down to achieve the standard self-indexing attachment history\cite{milnor2015lectures}. Our final 11-manifold $M$, derived from the code $\mathcal{C}$ is $D(ZQX^+)$, the double. 
Since contact between handles is a symmetric relation, turning $ZQX^+$ upside down and forming the double preserves bounded local geometry.
 This completes the description of $M$ as a smooth, or PL manifold, but in the next subsection we will clarify the Riemannian structure (based on the handle structure) on $M$, explain the significance of avoiding congestion in the construction of $M$, and finally prove that the size of the Riemannian (4,7)-systoles in $M$ governed, up to constants, by the minimal size of $X$, $Z$ logical operators, respectively.

As advertised in theorem \ref{generaltemplate}, the $k$-handle (and $k$-cell) structure of the
following three spaces are identical with the lifted quantum code when $k=3,4,5$: $ZQX, ZQX^+$, and $M^{11}$.
In particular,
$$H_k(ZQX;\mathbb{Z}) \cong
H_k(ZQX^+;\mathbb{Z}) \cong
H_k(M^{11};\mathbb{Z}) \cong H_{k-3}({\rm lifted \, code};\mathbb{Z}),\; {\rm for}\; k=3,4,5,$$
the ``extra" handles in our construction having indexes $1$ and $2$.  Seeking this separation between code
handles and ``extra" handles is what drove the construction to a dimension as high as $11$.

\subsection{Riemannian considerations}
The code $\mathcal{C}$ protects the logical qubits, the middle cohomology of line \ref{eq:chaincomplex} from $Z$ and $X$ errors. A good measure of this protection are the numbers $(d_Z, d_X)$ which denote the minimum size of an $Z$-logical and $X$-logical operator, respectively. In the manifold $M$ an $Z$-logical operator corresponds to a 4-cycle consisting of a union of some ``4-handle'' cores $\{Q_i\}$.  Let $d_Z$ denote the smallest number of $Q_i$ in such a union, representing an essential $4$-cycle in $H_4(ZQX;\Z)$ determined by the lifted $C$. If we speak of \emph{extended cores}, continuing the core after attachment through the lower index handles until it reaches the handlebody spine, then we may drop the quotes and say a (lifted) $Z$ logical operator is a cycle formed by no fewer than $d_Z$ extended 4-handle cores. A logical $X$-operator, in $M$, becomes what can be called a 4-cocycle or a 7-cycle. In the latter language it is a union of extended co-cores to the 4-handles, or equivalently due to the special structure of a double, the extended cores of a collection of 7-handles (see Figure \ref{fig:7handlecore}), dual to the ``4-handles'' on which it operates.

\begin{figure}[ht]
    \centering
    \begin{tikzpicture}[scale=1.2]
        \draw (-2,0) -- (2,0);
        \draw (0,0) ellipse (2 and 3);
        
        \draw (-0.2,0.6) to[out=-70,in=180] (0,0.4) to[out=0,in=250] (0.2,0.6);
        \draw (-0.2,0.2) to[out=70,in=180] (0,0.4) to[out=0,in=110] (0.2,0.2);
        \draw[fill=black] (0,0.4) circle (0.2ex);
        
        \draw (-0.2,-0.2) to[out=-70,in=180] (0,-0.4) to[out=0,in=250] (0.2,-0.2);
        \draw (-0.2,-0.6) to[out=70,in=180] (0,-0.4) to[out=0,in=110] (0.2,-0.6);
        \draw[fill=black] (0,-0.4) circle (0.2ex);
        
        \draw (-0.2,-0.8) to[out=-70,in=180] (0,-1) to[out=0,in=250] (0.2,-0.8);
        \draw (-0.2,-1.2) to[out=70,in=180] (0,-1) to[out=0,in=110] (0.2,-1.2);
        \draw[fill=black] (0,-1) circle (0.2ex);
        
        \draw (-0.2,-1.3) to[out=-70,in=180] (0,-1.5) to[out=0,in=250] (0.2,-1.3);
        \draw (-0.2,-1.7) to[out=70,in=180] (0,-1.5) to[out=0,in=110] (0.2,-1.7);
        \draw[fill=black] (0,-1.5) circle (0.2ex);
        
        \draw (-0.2,-1.8) to[out=-70,in=180] (0,-2) to[out=0,in=250] (0.2,-1.8);
        \draw (-0.2,-2.2) to[out=70,in=180] (0,-2) to[out=0,in=110] (0.2,-2.2);
        \draw[fill=black] (0,-2) circle (0.2ex);
        
        \draw (-0.2,-2.3) to[out=-70,in=180] (0,-2.5) to[out=0,in=250] (0.2,-2.3);
        \draw (-0.2,-2.7) to[out=70,in=180] (0,-2.5) to[out=0,in=110] (0.2,-2.7);
        \draw[fill=black] (0,-2.5) circle (0.2ex);
        
        \draw (-0.2,1.2) to[out=-70,in=180] (0,1) to[out=0,in=250] (0.2,1.2);
        \draw (-0.2,0.8) to[out=70,in=180] (0,1) to[out=0,in=110] (0.2,0.8);
        \draw[fill=black] (0,1) circle (0.2ex);
        
        \node at (0,2) {$\vdots$};
        \draw[fill=black] (0,3) circle (0.2ex);
        \draw[fill=black] (0,-3) circle (0.2ex);
        
        \node at (-0.5,3.1) {\scriptsize{11}};
        \node at (-0.5,-3.1) {\scriptsize{11}};
        \node at (0.8,3.1) {\scriptsize{index 11}};
        
        \node at (1,1) {\scriptsize{index 7}};
        \node at (1,0.4) {\scriptsize{index 6}};
        \node at (1,-0.4) {\scriptsize{index 5}};
        \node at (1,-1) {\scriptsize{index 4}};
        \node at (1,-1.5) {\scriptsize{index 3}};
        \node at (1,-2) {\scriptsize{index 2}};
        \node at (1,-2.5) {\scriptsize{index 1}};
        \node at (1,-3) {\scriptsize{index 0}};
        
        \node at (0.3,1.3) {$\scriptscriptstyle{4}$};
        \node at (-0.27,0.88) {$\scriptscriptstyle{7}$};
        \node at (-0.27,0.63) {$\scriptscriptstyle{5}$};
        \node at (0.25,0.15) {$\scriptscriptstyle{6}$};
        
        \node at (-0.27,-0.2) {$\scriptscriptstyle{6}$};
        \node at (-0.27,-0.55) {$\scriptscriptstyle{5}$};
        \node at (-0.27,-0.75) {$\scriptscriptstyle{7}$};
        \node at (-0.27,-1.05) {$\scriptscriptstyle{4}$};
        \node at (-0.27,-1.3) {$\scriptscriptstyle{8}$};
        \node at (-0.27,-1.6) {$\scriptscriptstyle{3}$};
        \node at (-0.27,-1.85) {$\scriptscriptstyle{9}$};
        \node at (-0.27,-2.1) {$\scriptscriptstyle{2}$};
        \node at (-0.33,-2.3) {$\scriptscriptstyle{10}$};
        \node at (-0.27,-2.6) {$\scriptscriptstyle{1}$};
        
        \node at (2.4,-1.7) {$ZQX^+$};
        
        \draw [decorate,decoration={brace,amplitude=10pt}](3,3) -- (3,-3) node [black,midway,xshift=3.5em] {$D(ZQX^+)$};
        \draw [decorate,decoration={brace,amplitude=10pt}](-3,1.8) -- (-3,-0.2);
        \node at (-4,1.3) {7-cycle of};
        \node at (-4,0.8) {$X$-logical};
        \node at (-4,0.3) {operator};
        \draw[->] (-2.5,0.8) -- (-0.6,0.85);
        \draw [decorate,decoration={brace,amplitude=10pt}](-3,-0.5) -- (-3,-2.5);
        \node at (-4,-1) {4-cycle of};
        \node at (-4,-1.5) {$Z$-logical};
        \node at (-4,-2) {operator};
        \draw[->] (-2.5,-1.5) -- (-0.6,-1.15);
    \end{tikzpicture}
    \caption{}\label{fig:7handlecore}
\end{figure}

The $k$-spine of a handlebody is built as the union of \emph{extended} cores of its handles of index
$\leq k$ (see Figs. \ref{fig1} and \ref{fig:mappingcylinder}).  Extended means that the cores, for $k>0$, are continued using
the defining product structure of the handles to which they attach until they meet cores of a lower index.
The \emph{$k$-spine} of a handlebody is a $k$-dimensional Whitney stratified space\footnote{A cell complex where cell attachments are, by maps which piecewise are smooth and of maximal rank, between a region on the sphere boundary and the cells to which it is attached.}.
 Similarly the \emph{dual $j$-spine} is the union of all extended co-cores of dimension $\leq j$ (corresponding to handles of index $\geq \dim - j$, $11-j$ in the present case). Without the index, ``spine'' and ``dual spine'' refer to the complete cell structure, up to, but not including, the top dimension. While the spines are dual spines of the handlebody $ZQX^+$ will have an increasing number of cells as we run through the family $\{\mathcal{C}\}$ of CSS codes, the LDPC property, together with the care we took with 1 and 2-handles, guarantees that the \emph{local geometric complexity} is bounded. Again, bounded means that all spines $S$ and dual spines $\hat{S}$ have Whitehead compatible triangulations within the ambient manifold $ZQX^+$. Then bounded local geometry is the condition that there is an $O(1)$ bound on the number of simplexes any given simplex meets, and on the biLipschitz comparison constants to $B^k$.

Let $H$ be a handlebody with $\de H \neq \varnothing$, spine $S$ and dual spine $\hat{S}$. Inductively, it is standard to place a mapping cylinder structure (MC) on $H$ by building up a map $f$ handle by handle, $f: \de H \ra S$ so that $H$ is (PL) homeomorphic $H \cong \operatorname{Map}(f) \coloneqq \de H \times [0,1] \slash_{h \times 1 \sim h^\pr \times 1 \iff f(h) = f(h^\pr)}$. If a top dimensional ball $B$ is removed from the interior of $M$ then there is the \emph{dual} mapping cylinder structure $H \cong \operatorname{Map}(\hat{f}) \coloneqq \de B \times [0,1] \slash_{b \times 1 \sim b^\pr \times 1 \iff \hat{f}(b) = \hat{f}(b^\pr)}$, where $\hat{f}$ is a map $\hat{f}: \de B \ra \hat{S}^+ \coloneqq \hat{S} \cup \de H$. Figure \ref{fig:mappingcylinder} demonstrates both cases with the punctured torus.

\begin{figure}[ht]
    \centering
    \begin{tikzpicture}[scale=1.6]
        \draw (-2.5,0) circle (1);
        \draw (-3.5,0) .. controls (-4.9,1.2) and (-2.8,2.4) .. (-2.3,0.98);
        \draw (-3.45,0.35) .. controls (-4.2,1) and (-3,1.8) .. (-2.65,0.98);
        \draw[color=white,line width=1.5mm] (-2.5,0) .. controls (-3.2,0) and (-3.5,0.2) .. (-3.7,0.4) .. controls (-4,0.8) and (-3.8, 1.2) .. (-3.5,1.4) .. controls (-3.2,1.6) and (-2.9,1.5) .. (-2.7,1.3) .. controls (-2.2,0.8) and (-2.2,0.3) .. (-2.5,0);
        \draw (-2.5,0) .. controls (-3.2,0) and (-3.5,0.2) .. (-3.7,0.4) .. controls (-4,0.8) and (-3.8, 1.2) .. (-3.5,1.4) .. controls (-3.2,1.6) and (-2.9,1.5) .. (-2.7,1.3) .. controls (-2.2,0.8) and (-2.2,0.3) .. (-2.5,0);
        \draw (-3,0.87) to [out=85,in=240] (-2.9,1.27);
        \draw (-2.73,0.97) to[out=85,in=250] (-2.7,1.05);
        \draw (-2.77,1.55) to[out=50,in=140] (-1.7,1.6) to[out=-40,in=55] (-1.6,0.45);
        \draw (-2.5,1.35) to[out=55,in=140] (-1.9,1.4) to[out=-40,in=45] (-1.8,0.7);
        \draw[color=white,line width=1mm] (-2.6,0.1) to[out=110,in=245] (-2.825,1.1);
        \draw[color=white,line width=1mm] (-1.6,1.25) to[out=290,in=15] (-2.3,0.125);
        \draw (-2.5,0) to[out=130,in=245] (-2.8,1.2);
        \draw (-2.6,1.45) to[out=45,in=110] (-1.6,1.25) to[out=290,in=15] (-2.5,0);
        \draw[fill=black] (-2.5,0) circle (0.2ex);
        \draw[->] (-3,-0.7) -- (-2.7,-0.1);
        \draw[->] (-2,-0.7) -- (-2.4,-0.1);
        \draw[->] (-1.7,-0.3) -- (-2,0.1);
        \draw[->] (-2.6,0.8) -- (-2.52,0.2);
        \draw[->] (-3.2,0.6) -- (-2.7,0.1);
        \draw[->] (-2.05,0.7) -- (-2.3,0.2);
        \draw[->] (-3.7,0.7) -- (-3.8,0.7);
        \draw[->] (-3.5,1.25) -- ++(120:0.1);
        \draw[->] (-2.85,1.27) -- ++(60:0.1);
        \draw[->] (-3.85,1.21) -- ++(-45:0.1);
        \draw[->] (-3.1,1.61) -- ++(-100:0.1);
        \draw[->] (-1.85,1.4) -- ++(45:0.1);
        \draw[->] (-1.68,1) -- ++(5:0.1);
        \draw[->] (-2.4,1.7) -- ++(-80:0.1);
        \draw[->] (-1.55,1.4) -- ++(220:0.1);
        \node at (-2.5,-1.3) {direct MC-structure};
        \node at (-1,0) {$S$};
        \draw[->] (-1.15,0.05) -- (-1.8,0.3);
        \draw[decorate,decoration={brace,amplitude=10pt}](-4,-1) -- (-4,2) node [black,midway,xshift=-1.5em] {$H$};
        
        \draw (2,0) circle (1);
        \draw (1,0) .. controls (4.5-4.9,1.2) and (4.5-2.8,2.4) .. (4.5-2.3,0.98);
        \draw (4.5-3.45,0.35) .. controls (4.5-4.2,1) and (4.5-3,1.8) .. (4.5-2.65,0.98);
        \draw (1.5,0.87) to [out=85,in=240] (4.5-2.9,1.27);
        \draw (4.5-2.73,0.97) to[out=85,in=250] (4.5-2.7,1.05);
        \draw (4.5-2.77,1.55) to[out=50,in=140] (4.5-1.7,1.6) to[out=-40,in=55] (4.5-1.6,0.45);
        \draw (2,1.35) to[out=55,in=140] (4.5-1.9,1.4) to[out=-40,in=45] (4.5-1.8,0.7);
        \draw (0.97,1.51) -- (1.1,1.25);
        \draw[white,line width=1mm] (1.6,0) to[out=180,in=-75] (0.7,0.7) to[out=105,in=220] (0.95,1.35);
        \draw[->] (1.6,0) to[out=180,in=-75] (0.7,0.7) to[out=105,in=220] (0.95,1.35);
        \draw[white,line width=1mm] (2.1,0.3) to[out=65,in=-55] (2.05,1) to[out=125,in=30] (1.1,1.4);
        \draw[->] (2.1,0.3) to[out=65,in=-55] (2.05,1) to[out=125,in=30] (1.1,1.4);
        \draw (2.7,1.67) -- (2.55,1.44);
        \draw[white,line width=1mm] (1.8,0.25) to[out=120,in=240] (1.68,1.1);
        \draw (1.8,0.25) to[out=120,in=240] (1.7,1.2);
        \draw[->] (1.9,1.45) to[out=60,in=145] (2.5,1.6);
        \draw[white,line width=1mm] (2.45,0.2) to[out=25,in=-35] (2.7,1.5);
        \draw[->] (2.35,0.15) to[out=20,in=-35] (2.7,1.5);
        \draw[->] (1.7,0.2) -- (1.3,0.7);
        \draw[->] (2.2,0.25) -- (2.5,0.8);
        \draw[->] (1.8,-0.25) -- (1.5,-0.8);
        \draw[->] (2.3,-0.2) -- (2.7,-0.7);
        \draw[->] (2,0.3) -- (1.9,0.95);
        \draw (2,0) ellipse (0.4 and 0.3);
        \node at (2,0) {$B$};
        \node at (2,-1.3) {dual MC-structure};
        \node at (0.2,1.7) {$\hat{S}$};
        \draw[->] (0.4,1.65) to[out=-50,in=190] (0.95,1.4);
        \draw[->] (0.4,1.65) to[out=45,in=135] (2.6,1.6);
        \draw[decorate,decoration={brace,amplitude=10pt}](3.4,2) -- (3.4,-1) node [black,midway,xshift=1.5em] {$H$};
    \end{tikzpicture}
    \caption{}\label{fig:mappingcylinder}
\end{figure}

\begin{lemma}
\label{admitsmetric}
    Let $H$ be a handlebody (or family of handlebodies) with $\de H \neq \varnothing$, with both spine and dual spine of bounded local complexity. Then $H$ admits a Riemannian metric $H_R$ with totally geodesic boundary so that the projections associated to both the direct and dual mapping cylinder structure are $O(1)$-Lipschitz. Furthermore, all sectional curvatures of the Riemannian metric are also $O(1)$ in absolute value. If the spine and dual spine have, instead, $O(\log(q))$ local complexity then the conclusions similarly become $O(\operatorname{polylog}(q))$.
\end{lemma}

\begin{proof}
    For the direct MC structure this is routine. The MC structures are built locally along with the Riemannian metric. As new handles are attached the metric must be adjusted near the attaching region to preserve the property of totally geodesic boundary. Locality and log-bounded geometry lead to the $O(\operatorname{polylog}(q))$ estimates.

    For building the dual MC-structure and controlling its Lipschitz constant there is an additional consideration. The deleted ball $B$ must be extensive. Obviously a small ball and large $\de H$ would force a divergent Lipschitz constant. In the context of our construction, $B$ may be chosen to be the 0-handle as specified (in building $X$) together with a subset of the 1-handles (within the ``4-handles'' associated to $Q$) sufficient to connect these 0-handles according to the pattern of a tree into one more extensive ball.
\end{proof}

Lemma \ref{admitsmetric} gives the handlebody $ZQX^+$ a Riemannian metric. Since the boundary is totally geodesic there is now a well-defined Riemmanian metric\footnote{Technically we should add that the metric has a totally geodesic product collar on its boundary so that the metric on the double is $\C^\infty$.} on the double $M_R \coloneqq D(ZQX^+)$. The following key lemma allows us to compare $\Z_2$-(4,7)-systoles in $M_R$ with $(d_Z,d_X)$ in $\mathcal{C}$. Let $S$ be the spine (i.e.\ cell structure) of $M_R$ constructed from the handle cores in $ZQX^+$ together with the co-cores of the upside-down copy.

\begin{lemma}\label{lm:47systolic}
    There are constants $c_k > 1$ depending on dimension $d$ and the local geometry of $S$, $1 \leq k \leq \dim(M) - 1$, so that if $\alpha$ is a $k$-cycle carried by smooth chains in $M_R$, $[\alpha] \in H_k(M;\Z_2)$, then there is a homologous cellular cycle $\lbar{\alpha}$ carried by $k$-cells of $S$ so that the following $k$-area bound holds:
    \[
        k\parea(\lbar{\alpha}) < c_k k\parea(\alpha)
    \]
\end{lemma}

\begin{proof}
    Denote the upside down, orientation reversed copy of $ZQX^+$ by $\lbar{ZQX^+}$.

We repeatedly apply Lemma \ref{push} below to move $\alpha$ off the $j$-cells, for $j>k$, while increasing the $k$-area
only by a bounded amount at each step, until
$\alpha$ is deformed to the $k$-spine.  At this point it is homologous to a cycle of $k$-cells from $k\operatorname{-spine}(M)$, and the lemma is proven, $c_k$ being the product of the $O(1)$ constants encountered along the way. 
Lemma \ref{push} assumes the geometry of a ball for each cell; we encounter additional $O(1)$ constants depending on the particular local geometry.
\end{proof}

The ``pushing" argument above is supplied by the following lemma.  See also Appendix \ref{ptbap} for an alternative approach.
\begin{lemma}
\label{push}
Let $c$ be a smooth $k$-chain in $B^n$, for $k<n$, such that $\partial c$ is contained in $\partial B^n=S^{n-1}$.  Here we regard $B^n$ as a standard ball in Euclidean space of radius $1$.  Then, there is some other smooth $k$-chain $d$ contained in $\partial B^n$ such that $\partial d=\partial c$ and such that the unsigned area of $d$ is bounded by a constant times the unsigned area of $c$.
\begin{proof}
The proof is based on ``pushing $c$ to the boundary".  
Given a point $p$ in $B^n$, define a map\footnote{Famously,  $f_p$ fails to be a function at $p$ where we interpret the image
$f_p(p)$ not to be a point but the entire $S^{n-1}$.  So, technically, $f_p$ is a closed relation, not a function.}
$f_p(\cdot)$ from $B^n$ to $S^{n-1}$ so that $f_p(x)$ is the point on $S^{n-1}$ contained in a ray starting at $x$ and passing through $p$.
We will choose $p$ uniformly from the ball of radius $1/2$ centered at the origin and estimate the expected area of $f_p(c)$.
We will bound this expected area by a constant times the area of $c$, implying that for some choice of $p$, choosing $d=f_p(c)$ has area bounded by constant times the area of $c$.

To estimate the expected area of $f_p(c)$, consider a $k$-simplex $\sigma$ with infinitesimal area $\dA$ located somewhere in $B^n$.  We bound the expected area of $f_p(\sigma)$ by a constant times $\dA$.  Bounding this for all $\sigma$ will bound the expected area of $f_p(d)$ by a constant times the area of $c$ by linearity of expectation, since $c$ can be approximated by a sum of infinitesimal simplices.

If $x$ is at distance $r$ from $\sigma$, then the solid angle of $\sigma$, as seen from $x$, is bounded by constant times $\dA/r^k$.
These rays cast some shadow on $\partial B^n$.  Since $x$ has distance at least $1/2$ from $\partial B^n$, all rays intersect the boundary at some angle strictly bounded away from the tangent to the boundary. 
 Hence, the shadow has area bounded by a constant times the solid angle, and hence bounded by constant times $\dA/r^k$.
Remark: a little geometry computes this angle exactly though it is not needed; consider a ray intersecting a point $a$ in $\partial B^n$, with the ray at angle $\theta$ away from the normal.  Suppose the ray just intersects the ball of radius $1/2$ at some point $b$.  Then a right triangle is formed containing vertices $a,b$ and the origin.  The hypotenuse has length $1$ and the side opposite to $\theta$ has length $1/2$ so $\sin(\theta)=1/2$ and $\theta=\pi/6$.  For any $\theta$ larger than this, the ray does not intersect the ball of radius $1/2$.

The probability that $x$ is within distance $r$ of $\sigma$ is bounded by constant times $r^d$.
Hence, integrating over $x$, for $d>k$ the expected area is bounded by constant times $\dA$.
\end{proof}
\end{lemma}

We have arrived at one of the goals of this section: Riemannian 11-manifolds $\{M\}$ modeling the CSS codes $\{\mathcal{C}\}$.

\begin{thm}
\label{powerZ2}
    There is a sequence $\{M_i\}$ of closed, simply connected, stably framed Riemannian $11$-manifolds exhibiting power-law $\Z_2$-(4,7)-systolic freedom. The power is $1+\alpha$, $\alpha = 1$ up to polylogarithmic factors.
\end{thm}

\begin{proof}
In the first verson of this paper, we proved the theorem for $\alpha=\frac{1}{4}$.  To show this,
    the power law CSS codes $\{\mathcal{C}\}$ from \cite{HHO} are translated first into chain complexes as in line \ref{eq:chaincomplex}, and sparsely lifted as described in section \ref{liftingsection}, and finally, as we have described in this section, into closed 11-manifolds $\{M\}$. The (4,7)-$\Z_2$-systolic properties of $M$ are shown in Lemma \ref{lm:47systolic} to mirror, up to constants, the $(d_Z, d_X)$ of the code $\mathcal{C}$. The intermediary between smooth singular chains in the manifold and the sizes of $X$-, $Z$-logical operators, are the cellular chains of $S$. Specifically, the singular 4 or 7 chains are related (with bounded distortion) to combinatorial $\Z_2$-cycles built from 4 or 7 cells of the spine $S(M)$.
    
    Using the codes of \cite{LD}, which can also be sparsely lifted in the same way, the power can be increased to $\alpha=1$.
\end{proof}

\subsection{Proof of Theorem \ref{thm:SRconstant}}
The key property of Whitehead's [Wh1940] triangulation of smooth manifolds is the bounded geometry of the simplexes. Achieving arbitrarily fine triangulations of charts while retaining a bound preventing simplexes from becoming too ``pointy'' was key in matching triangulations on overlapping charts. For us bounded geometry of the triangulation allows us to bound the multiplicative stretching of $k$-area as a smooth $k$-cycle is deformed into the $k$-skeleton.

It is not necessary to work with triangulations, piecewise smooth cellulations are fine. Also it is only necessary that the geometry of $j$-cells, $1 \leq j \leq d$, be uniformly bounded up to \emph{scale}; dilation is permitted.

\begin{definition}
    We say a piecewise smooth cell structure (or triangulation) on a smoothly Riemannian manifold $M^d$ has bounded geometry (up to scale), if
    \begin{enumerate}
        \item There is a uniform bound on how many cells meet any given cell, and
        \item There is a constant $C$, $C \geq 1$, so that there is a $C$-biLipschitz homeomorphism $h: b^\pr \twoheadrightarrow B^j$, onto the standard unit ball in Euclidean $j$-space, where $b$ is any cell (or simplex) of the cellulation (triangulation) with its induced Riemannian metric and $b^\pr$ denotes $b$ with metric rescaled by a positive constant.
    \end{enumerate}

    Similarly, for \emph{log bounded geometry}, the constants are replaced by $O(\operatorname{polylog}(q))$, $q$ the number of qubits.
\end{definition}

\begin{proof}[Proof of Theorem \ref{thm:SRconstant}]
    The top dimensional case of the fundamental $d$-cycle is routine. Standard techniques, such as Whitehead's, allow the construction of a subdivision where all cells (simplexes) have approximately uniform size as well as shape: each cell $b$ is $C$-biLipschitz homeomorphic to some fixed $B_\epsilon^j$ of radius $\epsilon>0$. Now, up to constants volume is counted by the number of $d$-cells.

    The case of a piecewise smooth singular $p$-cycle $\alpha$ (in our case, a cycle with $\Z_2$-coefficients though the argument is  the same as for integral coefficients) for $1 \leq p < d$ is more interesting. The outline is this: We perturb $\alpha$ to be transverse to the cell structure (triangulation) and observe that it passes through many $d$ cells. 
We then repeatedly apply lemma \ref{push}, increasing the area by $O(1)$ factors until we obtain some    
$\alpha_{\text{final}}$ is on the $p$-skeleton.  Replace it with its linearization $\lbar{\alpha}$, where $\lbar{\alpha}$ is the cellular (simplical) $p$-chain consisting of $k_i$ copies of the $i$th $p$-cell $b_i$, where $k_i$ is the homological degree of $\alpha_{\text{final}}$ over $b_i$. In our case $k_i = 0$ or $1 \in \Z_2$. This last step clearly reduces area (e.g.\ by a calibration argument): no $p$-cycle can cover a $p$-cell (simplex) more efficiently than the identity map. The result is
    \begin{equation}
        p\parea(\lbar{\alpha}) \leq p\parea(\alpha_{\text{final}}) \leq C^\pr p\parea(\alpha_{\text{initial}}) \tag{$\ast\ast$},
    \end{equation}
  completing the proof of Theorem \ref{thm:SRconstant}.
\end{proof}

\subsection{Some high-dimensional smooth topology}
\label{difftypesec}
\begin{thm}\label{thm:tangentbundle}
    Let $M$ be a closed\footnote{Closed means compact and without boundary.}, smooth, 11-dimensional manifold with stably framed tangent bundle and $\pi_1(M) \cong 0$, $H_i(M;\Z) \cong \Z$, $i = 0,4,7,$ and 11, and $H_i(M;\Z) \cong 0$, $i = 1,2,3,5,6,8,9,$ and 10. Then $M$ is diffeomorphic to $S^4 \times S^7 \# \Theta$, where $\Theta \in \boldsymbol\Theta_{11}$, the group of all 11-dimensional homotopy spheres. (According to \cite{mk63} $\boldsymbol\Theta_{11}$ has order 992.)
\end{thm}

\begin{proof}
    The stably framed hypothesis means that the direct sum of the tangent bundle with some trivial bundle $\epsilon^k$ is itself trivial:
    \begin{equation}
        \tau(M) \oplus \epsilon^k \cong M \times \R^{k+11}
    \end{equation}

    By Whitehead's Theorem, the first nontrivial homology group of a simply connected space is spherical, i.e.\ $\pi_4(M) \ra H_4(M;\Z) \ra 0$ is onto. Let $\alpha: S^4 \ra M$ represent a generator of $H_4(M;\Z)$. By general position we may assume that $\alpha$ is a smooth embedding. Consider the sum of tangent and normal bundles to $\alpha(S^4)$
    \begin{equation}
        \tau(\alpha(S^4)) + \nu(\alpha(S^4)) = \tau_M \vert_{\alpha(S^4)}
    \end{equation}

    Since the first and third bundles are stably trivial, so is the second. However, $\nu(\alpha(S^4))$ has fiber dimension 7, greater than the dimension (4) of the base, so the normal bundle is in the trivial range:
    \begin{equation}
        \nu(\alpha(S^4)) \cong S^4 \times \R^7
    \end{equation}

    The triviality of $\nu(\alpha(S^4))$ allows us to do framed surgery on $\alpha(S^4)$. That is, delete from $M$ a smooth $S^4 \times D^7$-neighborhood of $\alpha(S^4)$, and glue back a copy of $D^5 \times S^6$, since both pieces have boundary $S^4 \times S^6$. Since $\pi_4(\operatorname{SO}(7))$ is in the stable range, $7 \geq 4+2$, $\pi_4(\operatorname{SO}(7)) \cong \pi_4(\operatorname{SO}) \cong 0$, and the choice of gluing is unique. Denote the (closed) surgered manifold by $M^\pr$.

    It is easy to see that $M^\pr$ is some 11-dimensional homotopy sphere. Because the surgery was below codimension 2, the fundamental group is still trivial: $\pi_1(M^\pr) \cong 0$. Furthermore, the standard pair of Mayer-Vietoris sequences introduced in \cite{mk63} to compute the homological effect of surgery shows that the class $\alpha$ has been killed, capped off by $D^5 \times \mathrm{pt}$, and likewise its Poincare dual class in dimension 7 is punctured by the deletion of $\alpha(S^4)$ and also dies. The result is a simply connected homology sphere which by another celebrated theorem of Whitehead is a homotopy sphere, $\te$.

    $\te$ contains the framed 6-sphere $S^6 \times D^5$, just glued in. Now we need a lemma from high-dimensional knot theory.

    \begin{lemma}\label{lm:frameds6}
        A framed $S^6$, call it $\beta$, embedded in $S^{11}$, or any 11-dimensional homotopy sphere $\te$, is isotopically trivial in the sense that it is the boundary of a $D^7$ smoothly embedded in $\te$.
    \end{lemma}

    From the lemma, we can now understand $M$ by \emph{reversing} the surgery just performed. The reverse surgery on $S^6 \times D^5 \subset M^\pr$ recovers $M$, but by the lemma, we know that the $S^6$ is unknotted, i.e.\ isotopic to the standard inclusion $S^6 \subset D^{11} \subset \te$. Surgery on the standard $S^6 \subset S^{11}$ yields $S^4 \times S^7$, and similarly surgery on a standard $S^6 \subset D^{11} \subset \te$ yields $S^4 \times S^7 \# \te$. This is because the reverse surgery and the connected sum with the exotic sphere are \emph{disjoint} and therefore commute. This completed the proof of the theorem module the lemma.
\end{proof}

\begin{proof}[Proof of Lemma \ref{lm:frameds6}]
    The proof strategy is to first construct a concordance from $\beta$ to the standard embedding $\gamma: S^6 \subset D^{11} \subset \te$ and then apply Hudson's \cite{hudson69} Theorem: ``concordance implies isotopy'' which states that in codimension $\geq 3$ concordances are themselves isotopic to isotopies.

    This is the terminology: Given embeddings $f$, $g: P^p \hookrightarrow Q^q$ a concordance between them is an embedding $F: P \times I \ra Q \times I$, $I = [0,1]$ so that $F(P,0) = f(p) \times 0$ and $F(q,1) = f(q) \times 1$. An isotopy is a concordance $F$ with the additional property of being \emph{level preserving}, $F(p,t) \in Q \times t$, for all $t \in [0,1]$. We will work in a bounded context (also covered by \cite{hudson69}) where $\de P \neq \varnothing$ and the concordance we construct is already an isotopy (in fact constant) over $\de P$. Then \cite{hudson69} gives us an isotopy relative to $\de P$. The theorem \cite{hudson69} holds when the codomension $q - p \geq 3$, $q \geq 5$, and in all categories, Top, PL, and Diff. We work in Diff. The goal is to produce a concordance between $\beta$ and $\gamma$ (the standard $S^6$ unknot in $D^{11}$).

    The first step is to build a 7D normally framed manifold $N \subset D^{11} \times I$ with $\de N = \de_0 N \cup \de_1 N$ with $\de_0 N = \gamma(S^6) \times 0$ and $\de_1 N = \beta(S^6) \times 1$. Once this is done powerful tools from surgery theory are available to modify $N$ to a concordance $S^6 \times I \subset D^{11} \times I$.

    Given $p: S^6 \times D^5 \ra D^{11}$, let $(C,\de)$ denote its closed complement with $\de_-(C) = S^6 \times S^4$ and $\de_+(C) = \de D^{11} = S^{10}$. Consider the map $\delta$ classifying the 4th cohomology:
    \begin{center}
        \begin{tikzpicture}[scale=1.4]
            \node at (0,0) {$\de_- C$};
            \draw[->] (0.4,0) -- (1.6,0);
            \node at (1,0.15) {\scriptsize{$\pi$}};
            \node at (1.9,0.05) {$S^4$};
            \node at (0,0.5) {$\hookuparrow$};
            \node at (1.9,0.5) {$\hookuparrow$};
            \node at (0,1) {$C$};
            \draw[->] (0.4,1) -- (1.3,1);
            \node at (2,1) {$K(\Z,4)$};
            \node at (0.85,1.15) {\scriptsize{$\delta$}};
            \draw[dashed,->] (0.2,0.8) -- (1.6,0.2);
            \node at (0.9,0.65) {\scriptsize{$\epsilon$}};
            \node at (1.85,-0.7) {$F$};
            \draw[->] (1.85,-0.5) -- (1.85,-0.2);
        \end{tikzpicture}
    \end{center}

    $\pi$ is the projection: $\de_-(C) = S^6 \times S^4 \ra S^4$ and the right inclusion represents the 0- and 4-cell in the cell structure of $K(\Z,4)$. The homotopy fiber $F$ of this map is also included in the diagram as well as the dotted arrow $\epsilon$, representing the lifting problem. The only non-trivial obstruction to constructing $\epsilon$ lies in: $H_7(C, \de_- C; \pi_6(F)) \cong \Z_2$, as computed from the exact sequence of the fibration we have
    \begin{center}
        \begin{tikzpicture}
            \node at (0,0) {$\pi_7(K(\Z,4)) \ra \pi_6(F) \ra \pi_6(S^4) \ra \pi_6(K(\Z,4))$};
            \node[rotate=-90] at (-2.8,-0.5) {$\cong$};
            \node at (-2.8,-0.9) {0};
            \node[rotate=-90] at (1,-0.5) {$\cong$};
            \node at (1.1,-0.95) {$\Z_2$};
            \node[rotate=-90] at (3.5,-0.5) {$\cong$};
            \node at (3.5,-0.9) {0};
        \end{tikzpicture}
    \end{center}

    In fact, this potential $\Z_2$ obstruction vanishes as the map $\pi$ already provides a lifting over the 6-sphere, $S^6 \times \ast$.

    According to the Pontryagin Theorem construction, making $\epsilon$ transverse to the ``north pole'' of $S^4$ produces a normally framed 7D Seifert surfaces in $D^{11}$ with boundary $\beta(S^6)$. Thinking of $D^{11}$ now as $D^{11} \times 0$ we may push up a small round patch $P$ in $S$ up to the 1-level of $D^{11} \times I$ so that $S \setminus P$ becomes $N$, the desired normally framed cobordism.

    \begin{figure}[ht]
        \centering
        \begin{tikzpicture}[scale=1.1]
            \draw (-1,0) arc(0:-180:2 and 0.5);
            \draw[dashed] (-1,0) arc(0:180:2 and 0.5);
            \draw (-1,0) -- (-1,3);
            \draw (-5,0) -- (-5,3);
            \draw (-3,3) ellipse (2 and 0.5);
            \draw (-3,3) ellipse (0.2 and 0.15);
            \draw (-4,0) to[out=90,in=160] (-3.3,0.3) to[out=-20,in=200] (-2.5,0.3) to[out=20,in=90] (-2,0);
            \draw (-4,0) to [out=-90,in=170] (-3.7,-0.3) to[out=-10,in=200] (-3.4,-0.2) to[out=20,in=180] (-2.3,-0.3) to[out=0,in=-90] (-2,0);
            \draw (-4,0) to[out=90,in=-90] (-3.2,3);
            \draw (-2,0) to [out=90,in=-90] (-2.8,3);
            \draw (-2.5,1.3) arc (90:-90:0.1 and 0.2);
            \draw (-2.45,1.26) arc (90:270:0.07 and 0.16);
            \draw (-3.55,1.3) arc (90:270:0.1 and 0.2);
            \draw (-3.6,1.26) arc(90:-90:0.07 and 0.16);
            \node at (-1.5,0) {$\beta(S^6)$};
            \node at (-3,1.5) {$N$};
            \node at (-2,3) {$\gamma(S^6)$};
            \draw[decorate,decoration={brace,amplitude=10pt}] (-5.2,0) -- (-5.2,3) node [midway, xshift=-2.75em] {$D^{11} \times I$};
            \node at (-6,3) {$D^{11} \times 1$};
            \node at (-6,0) {$D^{11} \times 0$};
            
            \draw (1,1.5) -- (2.8,3) -- (5.8,1.2) -- (4,-0.3) -- cycle;
            \draw[fill=black] (1.9,2.25) circle (0.2ex);
            \node at (1.4,2.1) {$\gamma$};
            \node at (1.8,2.8) {$D^{11} \times 1$};
            \node at (4.7,2.4) {$D^{11} \times I$};
            \draw[fill=black] (4.9,0.45) circle (0.2ex);
            \node at (5.1,0.3) {$\beta$};
            \node at (6.3,1) {$D^{11} \times 0$};
            \draw (1.9,2.25) -- (3.9,1.05) to[out=-20,in=160] (4.35,1.1) to[out=-30,in=150] (4.45,0.45) to[out=-30,in=150] (4.9,0.45);
            \node at (4.7,1.2) {$N$};
            \node at (3.7,-0.7) {same picture dimensionally reduced};
        \end{tikzpicture}
        \caption{}
    \end{figure}

    Since $N$ is stably framed and of odd dimension $\geq 5$ there is no obstruction to doing simply-connected surgery $(L_7(\{e\}) = 0)$ to normally cobord $N$, rel boundary, to a product $N^\pr \cong S^6 \times I$, call the normal cobordism $\lbar{N}$. (For experts the target of this surgery problem is $(S^6 \times I, \de)$ and we construct a normal cobordism $W$ to the identity as illustrated in Figure \ref{fig:cobordismw}.)

    \begin{figure}[ht]
        \centering
        \begin{tikzpicture}[scale=1.1]
            \draw (1,1.5) -- (2.8,3) -- (5.8,1.2) -- (4,-0.3) -- cycle;
            \draw (1,1.5) -- (1,1.1) -- (4,-0.7) -- (4,-0.3);
            \draw (5.8,1.2) -- (5.8,0.8) -- (4,-0.7);
            \draw (1.9,2.25) -- (3.9,1.05) to[out=-20,in=160] (4.35,1.1) to[out=-30,in=150] (4.45,0.45) to[out=-30,in=150] (4.9,0.45);
            \draw (1.65,2.05) -- (3.7,0.8) to[out=-20,in=160] (4.15,0.825) to[out=-30,in=150] (4.25,0.2) to[out=-30,in=150] (4.65,0.25);
            \draw (1.9,2.25) -- (1.9,4.25) -- (4.9,2.45) -- (4.9,0.45);
            \draw (1.65,2.05) -- (1.65,4.05) -- (4.65,2.25) -- (4.65,0.25);
            \draw (1.65,4.05) -- (2.15,4.45) -- (5.15,2.65) -- (4.65,2.25);
            \draw[dashed] (2.15,4.45) -- (2.15,2.45);
            \draw (5.15,2.65) -- (5.15,0.65);
            \path[gray, pattern = north west lines] (1.65,2.05) -- (3.7,0.8) to[out=-20,in=160] (4.15,0.825) to[out=-30,in=150] (4.25,0.2) to[out=-30,in=150] (4.65,0.25) -- (4,-0.3) -- (1,1.5) -- cycle;
            \path[gray, pattern = north east lines] (1.65,2.05) -- (3.7,0.8) to[out=-20,in=160] (4.15,0.825) to[out=-30,in=150] (4.25,0.2) to[out=-30,in=150] (4.65,0.25) -- (4.65,2.25) -- (1.65,4.05) -- cycle;
            \node at (3.5,-1) {normal cobordism on $N$};
            
            \draw (7.5,0) -- (11.5,0) -- (11.5,0.5) -- (7.5,0.5) -- cycle;
            \draw (9,0.5) -- (9,2.5) -- (10,2.5) -- (10,0.5);
            \draw (9.5,2.5) -- (9.5,0);
            \draw[fill=black] (9.5,0) circle (0.2ex);
            \path[gray, pattern = north east lines] (7.5,0.4) -- (9.1,0.4) -- (9.1,2.5) -- (8.9,2.5) -- (8.9,0.6) -- (7.5,0.6) -- cycle;
            \path[gray, pattern = north east lines] (11.5,0.4) -- (9.9,0.4) -- (9.9,2.5) -- (10.1,2.5) -- (10.1,0.6) -- (11.5,0.6) -- cycle;
            \node at (9.5,-0.3) {$N$};
            \node at (9.2,3) {$A =$ ``available boundary''};
            \draw[->] (7.8,2.8) -- (8.7,2.3);
            \draw[decorate,decoration={brace,amplitude=10pt}] (7.2,0) -- (7.2,2.5) node [midway, xshift=-1.6em] {$Z$};
            \draw[decorate,decoration={brace,amplitude=10pt}] (6,2.5) -- (6,0);
            \node at (11,1.5) {$W$};
            \draw[->] (10.7,1.5) -- (9.6,1.5);
            \node at (12.5,0) {$D^{11} \times I$};
            \draw[->] (11.95,0) -- (11.6,0);
            \node at (11.5,2.2) {$S^6 \times D^5 \times I$};
            \draw[->] (11.5,2.4) to[out=150,in=30] (9.6,2.6);
            \node at (9.5,-0.8) {pictured even more};
            \node at (9.5,-1.2) {dimensionally reduced};
            \node at (0,0) {$\hspace{1em}$};
        \end{tikzpicture}
        \caption{}\label{fig:cobordismw}
    \end{figure}

    The right figure of Figure \ref{fig:cobordismw} in similar arguments has been called ``Browder's top hat'' after William Browder.

    Letting $Z$ denote the entire 13D normal bordism and $A$ (hatched) is ``available boundary,'' the boundary not touching our carefully built product $S^6 \times I$, or $D^{11} \times I \times 0$, we use that the submanifolds we are working with e.g.\ $N \subset D^{11} \times I$ have codimension $5 \geq 3$. This means that
    \begin{equation}
        \pi_1(A) \xrightarrow{\cong} \pi_1(Z)
    \end{equation}
    is an isomorphism (in fact both groups are trivial).

    This allows us to invoke the $\pi_1$-$\pi_1$ theorem (\cite{wall70} Ch.\ 9) to build a final 14-dimensional normal cobordism on $(Z, \de Z \setminus A)$ solving a surgery problem to the trivial model
    \[
        (D^{11} \times I \times I, D^{11} \times I \times 0\ \coprod\ S^6 \times I \times D^5 \times 1)
    \]
    The result is that $A$ is normally coborded to $A^\pr$, a product, $A^\pr \cong (D^{11} \setminus (\operatorname{int}(S^6 \times D^5))) \times I$. Thus, gluing back the very top piece of the top hat, $A^\pr \cup S^6 \times D^5 \times I$, becomes the sought concordance between $\beta$ and $\gamma$. This completes the proof of Lemma \ref{lm:frameds6} and therefore Theorem \ref{thm:tangentbundle}.
\end{proof}

In theorem \ref{powerZ2} we produced a sequence of appropriately scaled Riemannian manifolds $\{M_i\}$, $i = 1,2,3,\dots$ exhibiting algebraic-$\Z_2$-(4,7) systolic freedom, and with trivial fundamental group.
We do not know, however, that the manifolds obey the conditions on homology groups needed for theorem
 Theorem \ref{thm:tangentbundle}. 
 There are two issues.  First, while we know the homology groups over $\Z_2$ of the codes of \cite{HHO}, lifting to integers may give rise to torsion in the homology group.  Second, the codes of \cite{HHO} have a number of ``logical qubits" that grows with  the number of ``physical qubits" $n$ (the dimension of the middle dimensional space in
 \ref{eq:chaincomplex}).  Thus the $\mathbb{Z}_2$ Betti numbers $b_4(M;\mathbb{Z}_2),b_7(M;\mathbb{Z}_2)$ grow with $n$.
The rational Betti numbers
$b_i(M)$, $i\neq 0,2,4,7,9,11$, vanish as we have killed $\pi_1$ so $b_1(M)=b_{10}(M)$ vanishes.
Note that the Betti numbers $b_3$ and $b_5$ vanish since the input chain complex $E_2 \rightarrow E_1 \rightarrow E_0$ for these codes has
vanishing second and zeroth homology\footnote{This is a property of those specific codes, but in general nontrivial zeroth or second homology for that chain complex implies redundant stabilizers which can be dropped without changing the distance of the code.}.  However, in general it is possible that the lifted chain complex $\tilde E_2 \rightarrow \tilde E_1 \rightarrow \tilde E_0$ may have torsion in its first or zeroth homology so $M$ may have torsion in its third or fourth homology.
 
We conjecture that
there is some sparse lift of the codes of \cite{HHO,BE,LD} that has no torsion in the homology group.
Further, we conjecture that one can ``kill" unwanted homology by modifying the lifted code so that $H_4(M;\Z),H_7(M;\Z)\cong \Z$ without reducing the distance by more than polylogarithmic factors.  Note that ``killing" homology is something only a
topologist would wish for as it reduces the code's rate.
In section \ref{liftingsection}, we discuss this possibility further. 

It is not clear whether or not $b_2(M)$ vanishes for the codes of \cite{HHO,LD,BE}, and it may depend on the particular pairing of
$\frac{1}{2}$-edges chosen in our construction of $5$-cells.
However, even if $b_2(M)$ does not vanish\footnote{We thank the referee for pointing out that $b_2(M)$ might not vanish.}, it is possible to modify our construction so that $b_2$ vanishes, at the cost of worsening the power in algebraic systolic freedom.  Let $L$ be a handlebody which is a copy of all 0-, 1-, 2-handles in $ZQX^+$.
Recall the 1-handles in come from the multiplicity of balls removed from the $S^4$’s in building $X$ (in the construction of $QX$) and from the multiple connected components of the graphs $G$ removed from the $S^5$’s in building $Z$ (in the construction of $ZQX$); the 2-handles in $ZQX$ are dual to $b_1(G)$, and ``close the holes" left by deleting these graphs from the $S^5$’s\footnote{In detail, $H_1(G)$ is Alexander dual to $H^3(B^5\setminus G)$ which is Lefschetz dual to $H_2(B^5\setminus G, \partial(B^5\setminus G))$, a basis of which is spanned by the cores of the 2-handles.}, and $ZQX^+$ has additional $2$-handles to kill $\pi_1$.  Since $ZQX^+$ was built to be stably framed, $L$ has vanishing Stiefel-Whitney numbers and is in fact diffeomorphic to a standard model: a boundary connected sum of $S^1 \times D^{10}$’s and $S^2 \times D^9$’s. The only 3-handles in the description of $ZQX^+$ are from $X$ and they don't attach (at all) to the 2-handles, thus there are no relations: $H_2(ZQX^+)$ is freely generated by these 2-handles. Consequently the inclusion maps $H_2(L) \rightarrow H_2(ZQX^+) \rightarrow H_2(M)$ are isomorphisms. By theorem 2 of \cite{GG} we can embed the $2$-spine of $L$ in a $5$-ball $B^{11}$ with volume O$(N^{11/9})$.  Thicken the embedding of ${\rm spine}(L) \hookrightarrow B^{11}$, to obtain a copy of $L$ smoothly embedded in $B^{11}$. Now let $ZQX^{++}=(ZQX^+)\cup(B^{11})$, the union over $L$.  Gluing on this $11$-ball has the effect of killing $H_2$ (and no other modification to the homology). This allows us to create a variant, also denoted by $M$,  $M:=D(ZQX^{++})$ which now has nonvanishing Betti numbers only in dimensions 0, 4, 7, and 11.

Cellulate the $5$-ball $B^{11}$ with bounded local geometry.  The number of cells has increased from $O(N\,{\rm polylog}(N) )$ to $O(N^{11/9} \,{\rm polylog}(N))$, but using the codes of \cite{HHO} this still gives algebraic systolic freedom; the systolic freedom is even stronger using the codes of \cite{LD}.
Item 5 of Theorem \ref{generaltemplate} no
longer holds due to added cells in the $11$-ball.  Instead, the portion of the cellular chain
complex:
$${\rm span}_{\mathbb{Z}}(5-{\rm handles}) \xrightarrow{\rm attaching}
{\rm span}_{\mathbb{Z}}(4-{\rm handles}) \xrightarrow{\rm attaching}
{\rm span}_{\mathbb{Z}}(3-{\rm handles})$$
is isomorphic to a direct sum of
$\tilde E_2 \rightarrow \tilde E_1 \rightarrow \tilde E_0$ with a chain complex from the $11$-ball and so the lower bounds on systole and cosystole still hold, as any $4$-chain with vanishing boundary must have its boundary vanish in both terms in the direct sum separately.

\begin{prop}
\label{prop181}
If additional checks can be added to the code family of \cite{HHO} to make $b_1({\rm code})=1$
and if this can be lifted to a sparse torsion-free chain complex without reducing distance by more than polylogarithmic
factors, then there is a family of triangulations (appropriately scaled metrics) 
of
$S^4 \times S^7 \# \Theta$
exhibiting algebraic $\mathbb{Z}_2$-$(4,7)$ systolic freedom, where $\Theta \in \boldsymbol\Theta_{11}$.

In the case of triangulations the proposition does not have to mention the possibility of a homotopy sphere. In the PL category all homotopy spheres are standard\cite{stallings1960polyhedral}. 
\end{prop}.

From Theorem \ref{thm:tangentbundle} we know that up to diffeomorphism there are 992 possibilities for these manifolds.\footnote{If we wished here to work in the PL category Theorem \ref{thm:tangentbundle} would tell us that there is only one PL homeomoprhisms type, $S^4 \times S^7$, among the $\{M_i\}$. This is because the ambiguity in the smooth structure disappears PL: all homotopy spheres in dimension $\geq 5$ are PL homeomorphic to an $S^n$ [S][Z].} Already this tells us that \emph{some} one of these smooth manifolds would have a family of appropriately scaled Riemannian metrics exhibiting this variety of systolic freedom. The next theorem brings us closer to showing algebraic-$\Z_2$-(4,7)-systolic freedom exists on $S^4 \times S^7$ itself.

\begin{thm}\label{thm:sectioncurvature}
    Let $M^d$, $d \neq 4$, be an appropriately scaled Riemannian manifold, ($\vert \mathrm{section\ curvature} \vert$ $= O(1)$ and injectivity radius $\geq \Omega(1)$) and let $\operatorname{area}_p(M)$ denote the infimal area of an essential $\Z_2$-cycle in $M$, $1 \leq p \leq d$. There is a constant $c>1$ depending only on the dimension $d$ of $M$ so that if $\te$ is a smooth homotopy sphere also with an appropriately scaled Riemannian metric of volume one, then for all $1 \leq p \leq d$ there exists an appropriately scaled Riemannian metric on $M \# \te$ so that:
    \[
        \frac{1}{c}\operatorname{area}_p(M \# \te) \leq \operatorname{area}_p(M) \leq c \operatorname{area}_p(M \# \te)
    \]
\end{thm}

\begin{proof}
    For each dimension $d \neq 4$ there are only finitely many smooth structures (up to diffeomorphism) on $S^d$, for each one $\te$ pick an appropriately scaled Riemannian metric. The idea is to make a geometrically controlled connected sum:
    \[
        M \ra M \# \te
    \]
    and estimate, forwards and backwards, the $p$-area of one side from the other.

    Since both $M$ and $\te$ are already appropriately scaled we may assume they each contain a Euclidean $d$-ball of radius $=1$, and that the connected sum is formed by taking a tube of radius $\frac{1}{4}$, $S^{d-1}_{1/4} \times I$ and blending its ends into the annulus in $M$ and $\te$, parameterized in their Euclidean balls as: $B_1^d \setminus B^d_{1/2}$. Curvatures and injectivity radii remain controlled under the connected sum.

    \begin{figure}[ht]
        \centering
        \begin{tikzpicture}[scale=1.2]
            \draw (0.2,1.6) arc (0:-180:0.2 and 0.075);
            \draw[dashed] (0.2,1.6) arc (0:180:0.2 and 0.075);
            \draw (0.2,0) -- (0.2,1.5);
            \draw (0.2,0) arc (0:-180:0.2 and 0.1) -- (-0.2,1.5);
            \draw[dashed] (0.2,0) arc (0:180:0.2 and 0.1);
            \draw (0.2,0) to[out=-90,in=160] (0.4,-0.2);
            \draw (-0.2,0) to[out=-90,in=20] (-0.4,-0.2);
            \draw (0.2,0.2) arc(65:-242:0.45 and 0.3);
            \draw (0.2,0.45) arc(75:-255:0.75 and 0.55);
            \draw (0,-0.37) -- (0,-0.62);
            \draw (0.3,-0.3) -- (0.5,-0.5);
            \draw (-0.3,-0.3) -- (-0.5,-0.5);
            \draw (0.45,-0.075) -- (0.75,-0.075);
            \draw (-0.45,-0.075) -- (-0.75,-0.075);
            \draw (0.35,0.11) -- (0.56,0.3);
            \draw (-0.35,0.11) -- (-0.56,0.3);
            \draw (-0.3,0.42) to[out=200,in=90] (-2,-1) to[out=-90,in=180] (0,-2) to[out=0,in=-90] (2,-1) to[out=90,in=-20] (0.3,0.42);
            \node at (0,-1.2) {$M$};
            
            \draw (0.2,1.5) to[out=90,in=205] (0.37,1.81);
            \draw (-0.2,1.5) to[out=90,in=-25] (-0.37,1.81);
            \draw (0.27,1.7) arc (-50:235:0.42 and 0.2);
            \draw (0.31,1.5) arc (-63:254:0.7 and 0.4);
            \draw (0.31,1.5) arc (-63:-72:0.7 and 0.4);
            \draw (0.42,1.85) -- (0.68,1.85);
            \draw (-0.42,1.85) -- (-0.69,1.85);
            \draw (0.3,2) -- (0.4,2.18);
            \draw (-0.3,2) -- (-0.4,2.18);
            \draw (0,2.07) -- (0,2.25);
            \draw (0.35,1.75) -- (0.5,1.6);
            \draw (-0.35,1.75) -- (-0.5,1.6);
            \draw (0.3,1.5) to[out=20,in=-90] (2,3) to[out=90,in=0] (0,4) to[out=180,in=90] (-2,3) to[out=-90,in=160] (-0.3,1.5);
            \node at (0.,2.8) {$\theta$};
            
            \node at (2.3,1) {annuli};
            \draw[->] (1.8,0.95) -- (0.6,1.5);
            \draw[->] (1.8,0.95) -- (0.6,0.4);
            \node at (-2.3,1) {blending regions};
            \draw[->] (-1,1) to[out=0,in=245] (-0.3,1.65);
            \draw[->] (-1,1) to[out=0,in=105] (-0.3,-0.1);
            \node at (0.9,1) {tube};
            \draw[->] (0.5,1) -- (0.3,1);

            \node at (3.5,0) {$\hspace{0.5em}$};
        \end{tikzpicture}
        \caption{}
    \end{figure}

    Because $\te$ is acyclic any essential $p$-cycle $\delta$ in $M \# \te$ can be cut off along a coordinate-slice sphere $S_r$, $\frac{1}{2} < r < 1$, in the annulus $(B_1 \setminus B_{1/2}) \subset M$ and cut off by a well-chosen cone in $S_4$ (see Lemma \ref{push}) without making it trivial. Similarly, a $p$-cycle $\delta$ in $M$ can be modified to lie in $M \# \te$ again by cutting it off along a well-chosen $S_r$, $\frac{1}{2} < r < 1$, and closing it by a cone also in $S_r$.

    These operations, cutting and patching in a cone, must respect $p$-area up to a multiplicative factor. This is accomplished by 1) using the co-area formula (page 24, \cite{cm11}) to choose $S_r$, $\frac{1}{2} < r < 1$, so that $(p-1)\parea(\delta \cap S_4)$ is no more than comparable to $p\parea(\delta)$, and 2) using the probabilistic argument of Lemma \ref{push} or Lemma \ref{pushalt} to select the cone point $q \in S_r$ so that:
    \begin{equation}
        p\parea(\operatorname{cone}_q(\delta \cap S_r)) \leq \pi (p-1)\parea(\delta \cap S_r)
    \end{equation}
\end{proof}

The conclusion of Theorem \ref{thm:tangentbundle}, Proposition \ref{prop181}, and Theorem \ref{thm:sectioncurvature} is

\begin{thm}
If additional checks can be added to the code family of \cite{HHO} to make $b_1({\rm code})=1$ 
and if this can be lifted to a sparse torsion-free chain complex without reducing distance by more than polylogarithmic
factors,
    $S^4 \times S^7$ admits a sequence of appropriately scaled Riemannian metrics exhibiting algebraic-$\Z_2$-(4,7)-systolic freedom.
\end{thm}

\subsection{Reducing the Dimension}
We close this section with some speculation on whether it is possible to reduce the dimension to smaller than $11$.  We have obtained such a large dimension partly because we enforced a certain ``separation" in dimension: $X$-stabilizers, qubits, and $Z$-stabilizers correspond to $3-$, $4-$, and $5-$``cells" and also dually to $6-$, $7-$, $8-$``cells", with $6>5$.  It may be possible to reduce the dimension if this separation is not enforced, though we have no specific construction.

Another possible way to reduce the dimension is to consider building manifolds from classical codes, i.e., codes with only one type of stabilizers, say $X$-stabilizers.  In this case, it may be possible to map $X$-stabilizers to $2$-cells and map qubits (though now perhaps we should say ``bits" instead) to $3$-cells and build a $7$-manifold.  This may give interesting constructions of manifolds using the rich literature on classical codes.  Further, since the codes of \cite{HHO} are, in a sense, constructed as circle bundles over a base that is a classical code, it may be possible to construct them as circle bundles over $7$-manifolds; we leave this for future work.

\section{Lifting}
\label{liftingsection}
In this section, we further consider the question of lifting.
Our main results are that every complex over $\Z_2$ admits some lift, and that the codes of \cite{HHO} admit a sparse lift.

Throughout this section, consider a chain complex $\cA$ with vector spaces $\cA_j$ and boundary operators $\partial_j$ from $\cA_j$ to $\cA_{j-1}$.  We regard these as vector spaces over $\F_2$.
The spaces $\cA_j$ have preferred bases, and basis vectors are termed cells; in the case of a quantum code, we have a three term complex (meaning, three different vector space), with cells corresponding to $Z$-stabilizers, qubits, and $X$-stabilizers respectively.  

In this section, the word ``cell" simply refers to an algebraic object: a basis vector for some vector space.  No assumption is made, for example, that a $1$-cell has only two $0$-cells in its boundary.

We say that a lift of the boundary operators is {\it admissible} if the square of the lifted operators vanishes.
Thus, a lift of a chain complex as defined below is an admissible lift of the boundary operators.

Since we will talk about complexes with different numbers of types of cells, we will use the term $k$-complex to refer to a complex with $j$-cells for $0\leq j \leq k$.  We emphasize that all complexes are purely algebraic objects: for example, a $1$-cell may have arbitrarily many $0$-cells in its boundary, and it is possible for the zeroth Betti number of the complex to vanish.

\subsection{Every Complex is Liftable}
Our first result is, however, that such a lift does always exist, if we do not worry about preserving sparsity. 
To give a precise definition of sparsity, say that a chain complex over $\F_2$ is $s$-sparse if every row and every column of every boundary operators has at most $s$ nonzero entries.  Say that a chain complex over $\mZ$ is $s$-sparse if every row and every column of every boundary operators has the sum of absolute values of its entries bounded by $s$.  One useful notation for this is to let the $\ell_1$ norm $|\cdot|_1$ denote the sum of absolute values of entries of a vector, and then the $\ell_1$ norm of every row and column vector must be bounded by $s$.

When we refer to preserving sparsity, we have in mind a sequence of chain complexes parameterized by some integer $n$ with the dimension of the vector spaces in the $n$-th complex bounded by $O(n)$.  When we say that $\cA$ is sparse, we mean that there is a sequence of complexes, all of which are $O(1)$-sparse.  When we say that a lift is sparse, we mean also that the lifted complexes are $O(1)$-sparse.  A weaker, but still desirable, condition might be that the lift is log-sparse, i.e., that the lifted complexes are $O(\log(n))$ sparse.

We will make several conjectures in this section.  All of these conjectures are of the form: for some sparse object $A$, there does not exist some other sparse object $B$.  This should always be interpreted as a conjecture about families, where $A$ is a family of $O(1)$-sparse objects.  The conjecture that the sparse $B$ does not exist in its weakest form is simply that no $O(1)$-sparse $B$ exists, but a stronger conjecture would be that no log-sparse or even $o(n)$-sparse $B$.

\begin{lemma}
\label{liftexists}
Every chain complex $\cA$ over $\F_2$ has a lift to some $\tilde \cA$ over $\mZ$.
Further, the lift can be done so that $\tilde \cA$ has no torsion in its homology or cohomology and has the same Betti numbers as does $\cA$, i.e., $H_j(\tilde \cA)$ and $H^j(\tilde \cA)$ are both isomorphic to $\mZ^{b_j(\cA)}$, where $b_j(\cA)$ are the $\F_2$ Betti numbers of $\cA$.
\begin{proof}
Very briefly, the proof is that by a sequence of row and column operations over $\F_2$, we can bring the boundary operators $\partial_j$ to a form in which the naive lift gives a chain complex without torsion.  Then, by undoing the lifted row operations, we obtain a chain complex $\tilde \cA$ that lifts $\cA$.

In detail,
consider over $\F_2$ the elementary operation of replacing some basis element of $\cA_j$ by the sum of that basis element and some other basis element.  This acts as an elementary column operation on $\partial_j$ and as an elementary row operation on $\partial_{j+1}$.
By a sequence of these operations, we can bring the operators $\partial_j$ to a form in which each operator has at most one nonzero element in each row and column.  Up to permutation of basis elements, which can also be implemented by row and column operations, this means that the $j$-th boundary operator is in the form of the following block matrix
$$D_j \equiv \begin{pmatrix} 0 & 0 & 0 \\ I & 0 & 0 \\ 0 & 0 & 0\end{pmatrix},$$
so that the first block of each $\cA_j$ is the range of the second block of $\cA_{j+1}$ under $\partial_{j+1}$, while the third block represents homology.

This means that $\partial_j = R_j D_j R_{j+1}^{-1}$, where $R_j$ is a product of elementary row operations.  
Lifting each elementary row operation by the naive lift and taking the product of those lifts gives a lift of $R_j$.  The lift $\tilde R_j$ is invertible over the integers since it is a product of matrices with determinant $1$, and its inverse is a lift of $R_j^{-1}$.  Of course, one may construct the inverse of $\tilde R_j$ by the product of the inverses of the lifted elementary row operations.

Lift $D_j$ by the naive lift to give $\tilde D_j$.  Note that $\tilde D_j \tilde D_{j+1}=0$.  Then $\tilde \partial_j = \tilde R_j \tilde D_j \tilde R_j^{-1}$ defines the chain complex $\tilde \cA$.
\end{proof}
\end{lemma}

Lemma \ref{liftexists} gives a lift of $\cA$ without torsion.  However, the resulting complex may not be sparse.  We next consider this question.

\subsection{Sparse Lifts of the Fiber Bundle Code (and Codes of \cite{LD,BE})}
First, it is worth noting that the fiber bundle code of \cite{HHO} has a lift to $\mZ$ which is sparse (indeed, it has the same sparsity as the complex over $\F_2$ does).  The code in that paper is constructed as a fiber bundle whose base is a classical LDPC code.  Adding a sign in a fairly obvious way (simply a sign needed in a boundary operator for a product complex over an arbitrary coefficient group) gives a lift of that complex, as we now explain.
The technique we give here also gives a sparse lift of the codes of \cite{LD,BE}; see section IVC of \cite{BE} which shows that the codes of \cite{LD} are fiber bundle codes with circle as fiber.

Recalling the definition in that paper briefly, the construction is a $2$-complex $\cE$.
The complex $\cE$ is a twisted product of two $1$-complexes $\cB,\cF$ called the base and the fiber.
The twisted boundary operator is defined a follows.
Acting on cell in $\cE$ which is a product of a $0$-cell $b^0$ in $\cB$ times an arbitrary cell $f$ in $\cF$ the boundary is given by: 
\be                                                   
\partial^\cE_{(0,q)} (b^0 \otimes f) = b^0 \otimes \partial^\cF f,
\ee
where the subscript $(0,q)$ indicates that the cell is a product of a $0$-cell in $\cB$ with a $q$-cell in $\cF$ for some $q$.
On the other hand, on a cell which is a product of a $1$-cell $b^1$ in $\cB$ with some $f$, the boundary is
\be
 \partial^\cE_{(1,q)} (b^1 \otimes f) =                                                                              
  b^1 \otimes \partial^\cF f \,\,                                                                                             
  +          \sum_{a^0 \in \partial b^1} a^0 \otimes \varphi(b^1,a^0) f,                                                             \label{bdef}
 \ee
 where $\varphi(b^1,a^0)$ is some automorphism of the fiber which acts as a permutation of the cells of the fiber and the sum
 is over $0$-cells $a^0$ in the boundary of $b^1$.
The specific code considered in that paper depends on specific choices of fiber, twists, and base, but this definition is general.
Over $\mZ$, we can replace Eq.~(\ref{bdef}) with
\be
 \tilde \partial^\cE_{(1,q)} (b^1 \otimes f) =     
b^1 \otimes \tilde \partial^\cF f \,\,-
\sum_{a^0 \in \partial b^1} a^0 \otimes \varphi(b^1,a^0) f,                                                             \label{bdefZ}
 \ee
 to get a sparse lift of the chain complex, where $\tilde \partial^\cF$ denotes a lift of the boundary operator $\partial^\cF$.  Any lift $\tilde \partial^\cF$ can be used as input to Eq.~(\ref{bdefZ}); in particular, the naive lift of $\partial^\cF$ suffices.  However, later when considering homology of the lifted complex, we explain a more natural choice of lift $\tilde \partial^\cF$.

 \subsection{Sufficient Conditions for a Sparse Lift}
 Let us now give a sufficient condition to have a sparse lift of a sparse $2$-complex, i.e., a three term complex.
 We fix some sparse lift of $\partial_1$ to $\tilde \partial_1$; for example, we may take the naive lift.
 Then, we will consider whether there is a sparse lift of $\partial_2$ so that the lifted boundary operators are admissible.

Label $2$-cells by an index $z$.  Label $1$-cells by index $q$ and label $0$-cells by index $x$.
For each $2$-cell $z$, regard $\partial\{z\}$ as a set: the set of $1$-cells in the boundary of $z$.  The brackets around $z$ are notation intended to indicate that this is a set.  Let $\partial \partial \{z\}$ be the set of $0$-cells which are in the boundary of some $1$-cell in the boundary of $z$.  Let $A_z$ be a matrix over $\mZ$, whose rows correspond to $0$-cells in 
 $\partial \partial \{z\}$ and whose columns correspond to $1$-cells in $\partial \{z\}$, and whose entries are the matrix elements of $\tilde \partial_1$ between the corresponding cells.
 Since $\cA$ is sparse, $A_z$ is an $O(1)$-by-$O(1)$ matrix.
 
 Let $\vec 1_z$ denote the all $1$s vector on the domain of $A_z$.  By assumption that $\cA$ is a chain complex, $A_z \vec 1_z$ equals $0$ mod $2$.
 If, for all $z$, we had $A_z \vec 1_z=0$ identically, then the naive lift of $\partial_2$ would give an admissible lift of the boundary operators.
 
 Suppose instead for some $z$ that $A_z \vec 1_z \neq 0$.  We now consider this case.
 
Say that a matrix $A$ over $\mZ$ ``does not have $2$-torsion" if, taking the domain to the vectors of integers, for any vector $v$ in the range of $A$ such that $v$ is divisible by $2$, the vector $v/2$ is also in the range of $A$.
 
 We claim that:
 \begin{lemma}
 If $\cA$ is sparse, and all $A_z$ do not have $2$-torsion for the given $\tilde \partial_1$, then there is an admissible sparse $\tilde \partial_2$.
 \begin{proof}
 For each $z$, we will construct some vector $v_z$ in the domain of $A_z$ which is equal to $\vec 1_z$ mod $2$, and such that $A_z v_z=0$, and such that $|v_z|_1$ is $O(1)$.
 Then, given these vectors, we define the lift $\tilde \partial_2$ as follows: for each $2$-cell $z$, the image of that $2$-cell under $\tilde \partial_2$ is the sum of $1$-cells in $\partial\{z\}$ with coefficients given by $v_z$.
 
Ignoring the requirement that $|v_z|_1$ be $O(1)$, the existence of such a vector follows by assumption: the vector $A_z \vec 1_z$ is divisible by $2$, so there is some vector $w$ such that $A_z w = (A_z \vec 1_z)/2$.  Then set $v_z=\vec 1_z-2w$.

To bound $|v_z|_1$, we give a simple non-explicit argument.  The matrix $A_z$ is an $O(1)$-by-$O(1)$ matrix with entries that are $O(1)$.  So, there are only $O(1)$ such matrices, i.e., we have some fixed, finite set of matrices, and for all $n$, all matrices $A_z$ are chosen from this set.  For each matrix in the set, there is some $w$ with $A_z w = (A_z \vec 1_z)/2$ that minimizes $|w|_1$.  Since the set is finite, the maximum of $|w|_1$ over $A_z$ in this set is some finite number, independent of $n$.

A more explicit way to bound the minimal $|w|_1$ is to bring the matrix $A_z$ to a Hermite normal form by Gaussian elimination, using the absence of $2$-torsion to show that none of the pivots are divisible by $2$, and bounding the matrix elements that appear in this process.  We do not consider this further.
 \end{proof}
 \end{lemma}
 
 This condition of no $2$-torsion may not always hold.  While there may be other ways to do the lift, for example by choosing different lifts of $\partial_1$, we make the following conjecture:
 \begin{conjecture}
There exist sparse chain complexes with $\F_2$ coefficients that do not have any sparse lift to integer coefficients.
Further, there exist sparse $2$-complexes with $\F_2$ coefficients that do not have any sparse lift to integer coefficients.
\end{conjecture}

Even more strongly, we may consider whether there is a sparse ``stable lift".  Here a stable lift means that starting with the original chain complex $\cA$, we stabilize by adding a finite number of additional pairs of $j$- and $j+1$-cells, for various $j$.  Each pair added is a pair of cells $b_j,b_{j+1}$, so that $\partial b_{j+1}=b_j$, $\partial b_j=0$, $\partial^T b_{j+1}=0$.  We then lift that stabilized complex.  We conjecture
\begin{conjecture}
There exist sparse $2$-complexes with $\F_2$ coefficients that do not have any sparse stable lift to integer coefficients.
\end{conjecture}

\subsection{Lifting Without Torsion}
Even if a sparse lift of some chain complex exists, the lift may introduce torsion in homology.  This phenomenon can be seen already at the level of a $1$-complex, which has only two types of cells, $0$-cells and $1$-cells.
Such a complex is defined by a single matrix $\partial$ over $\F_2$.

Since we are considering a $1$-complex, any lift $\tilde \partial$ is admissible: there is only one boundary operator so the condition $\partial^2=0$ is vacuous.
However, the question arises whether the lift can avoid having torsion.  If $\partial$ is full rank, the lift has no torsion in homology iff ${\rm det}(\tilde \partial)=1$.  Of course, ${\rm det}(\tilde \partial)=1$ mod $2$, so the lift cannot introduce $2$-torsion but may introduce $p$-torsion for $p>2$.

We can avoid introducing torsion if $\partial$ has a sparse LU decomposition, $\partial=PLUQ$ for some permutation matrices $P,Q$ and sparse lower and upper triangular $L,U$.  Then, the diagonal entries of $L,U$ are equal to $1$, and the naive lift of those matrices gives a matrix with determinant $1$: $\tilde \partial=\tilde P \tilde L \tilde U \tilde Q$.
If $\tilde L,\tilde U$ are sparse, then $\tilde \partial$ is sparse.
However, it is not clear if such a sparse LU decomposition exists.

We conjecture:
\begin{conjecture}
There exist sparse matrices with no sparse LU decomposition.  Further, with high probability, a sparse matrix chosen uniformly at random has this property.
\end{conjecture}

\begin{conjecture}
There exist sparse chain complexes over $\F_2$ that have a sparse lift to integer coefficients, but every sparse lift
to integer coefficients has some torsion in its homology and cohomology.  Further, there exist sparse $1$-complexes with trivial homology and cohomology, but every sparse lift to integer coefficients has some torsion in its homology and cohomology. 
\end{conjecture}

\subsection{Lifting Without Torsion for the Bundle Code: Conjectured Proof Sketch}
We have shown above that the fiber bundle code admits a sparse lift to integer coefficients.
Here we briefly sketch a proposed plan of proving that there is a sparse to integer coefficients such that the lifted complex
has no torsion in its homology and cohomology and has the same Betti numbers as does the complex over $\F_2$.

We then further discuss questions of killing unwanted homology (i.e., modifying the code so that it only has $1$ logical qubit), as an attempt to prove conjecture \ref{goodlift}.

The lift that we propose is to use Eq.~(\ref{bdefZ}), choosing however a particular lift for $\tilde \partial^\cF$.  Rather than choosing the naive lift, we choose a lift with signs: all entries are $\pm 1$.  The particular lift chosen is the ``obvious" one: the fiber is a cellulation of the circle, and so we lift it to a cellulation of the circle with integer coefficients.  Each $1$-cell has two $0$-cells in its boundary, and in the lift the boundary of a $1$-cell is the sum of those $0$-cells with opposite coefficients.

The base is a $1$-complex.  In general, we do not know if every $1$-complex has a sparse lift without introducing torsion.  However, in this case we believe that a sparse lift of the base exists, and that the naive lift suffices.  Before explaining this let us clarify which base we mean: that paper\cite{HHO} considered primarily a base that was polylogarithmically sparse, and then explained a weight reduction procedure for the base.  Here we will just explain the case of a polylogarithmically sparse base; it seems that if we can (polylogarithmically) sparsely lift the base without torsion, then we can also O(1)-sparsely lift the weight reduced base without torsion, by applying a similar weight reduction procedure to the lift.

The base then is defined by a boundary operator which is a $(3/4)n$-by-$n$ matrix for some $n$.  That is, there are many more $1$-cells than $0$-cells.  With high probability, this matrix is full rank.  
So, for some matrix of size $(3/4)n$-by-$(3/4)n$, the matrix of full rank.  Indeed, there are, with high probability, many such matrices.  If any one of those matrices was unimodular, the lift would have no torsion.  More strongly, if we consider all possible 
$(3/4)n$-by-$(3/4)n$ submatrices which are full rank, if the gcd of the determinants of these submatrices is $1$ then
there is no torsion in the lift.  We conjecture that, given that there are a large number of possible such submatrices, this gcd will be $1$ with high probability.

If this holds, then the base code lifts without torsion with high probability.  Then, by general homological arguments, the fiber bundle code constructed in that paper will also lift without torsion.

Even if this works up to here, this is still not sufficient.  The resulting manifold would have
$H_i(M;\Z) \cong \Z$, $i = 0,11$, and $H_i(M;\Z) \cong 0$, $i = 1,2,3,5,6,8,9,$ and 10, but it would have
$H_4(M;\Z),H_7(M,\Z)\cong Z^{n/4}$.  It is necessary then to kill the unwanted homology.  It is unclear whether or not this can be done without hurting the distance or number of qubits of the code.  In the language of manifolds considered in this paper, we can add some large (i.e., size growing with $n$) $5$-cells whose boundary kills the unwanted homology, but since the size of the cell is growing with $n$, it may hurt the distance or number of qubits of the code.

It may alternatively be possible to kill the unwanted homology by modifying the base code.  For example, one can add additional checks which constrain that certain bits are equal to zero (adding an additional $0$-cell which is in the boundary of exactly one $1$-cell).  This can be done to remove homology.  However, doing this may unfortunately re-introduce torsion (essentialy, the gcd argument above no longer will apply), and also it may impact the distance of the quantum code since the change in the base code changes both the number of qubits and the checks of the quantum code.

We leave these coding questions for future work.  These are somewhat surprising questions from the point of view of coding theory.  For one thing, here we are struggling to {\it reduce} the number of logical qubits of the code, to realize systolic freedom on a single smooth manifold $S^4 \times S^7$.  In general, in coding theory, one would of course like to have the number of logical qubits as large as possible.

\subsection{Lifting Hypergraph Product Codes}
A further application is to hypergraph product codes\cite{tillich2013quantum}.  These codes have linear rate (i.e., rank of middle dimensional homology of the $\Z_2$ chain complex associated with the code is proportional to the number of qubits) and distances $d_X,d_Z$ both proportional to the square-root of the number of qubits.

These codes have a sparse lift as follows.  The chain complex $\mathcal{C}$ of the code is the homological product of two chain complexes, that we will write $\cA,\cB$.  Both $\cA,\cB$ are $1$-complexes, with cells corresponding to vertices and hyperedges of a hypergraph.  Since they are $1$-complexes, the naive lifts give sparse lifts of $\cA,\cB$.  Then, the product of these naive lifts gives a sparse lift of $\mathcal{C}$.  Note that the product of the naive lifts is not the naive lift of the product: there are extra minus signs.

This then gives a homological code\footnote{A homological code is a quantum code with qubits identified with cells of a particular dimension and $X$- and $Z$-stabilizers identified with cells of one lower and one higher dimension on some cellulation of a manifold.} with linear rate and square-root distance (up to polylogs if we want a simply connected manifold).  The first example of a homological code with linear rate and with distance growing as a polynomial of the number of qubits used hyperbolic $4$-manifolds\cite{guth2014quantum}, answering a question of Z\'{e}mor\cite{zemor2009cayley} about the existence of such codes;  in that construction the distance grew more slowly than square-root.  Our construction erodes the distinction between a homological code and a CSS code: sparsely liftable CSS codes can be turned into homological codes.

{\it Acknowledgments---} The first named author would like to thank Xin Zhou for discussions on the isoperimetric inequality appearing in Appendix B.   
\newpage
\bibliography{refs}

\appendix
\section{Cycle Bases and Decongestion Lemma}
\label{cyclebasis}
Given a graph $G$, a {\it cycle basis} is a set of simple cycles\footnote{A simple cycle is a closed path on a graph that does not repeat any vertices.  Every simple cycle defines a closed $1$-chain with $\Z_2$ coefficients, taking a coefficient $1$ on all edges in the cycle and $0$ on other edges.  However, not all closed $1$-chains correspond to simple cycles.  Sometimes in graph theory a simple cycle is just called a cycle, but we use the term simple cycle to avoid confusion.} such that the corresponding $1$-chains form a basis for $H_1(G;\Z_2)$.

A {weakly fundamental cycle basis} is a cycle basis which can be linearly ordered so that every cycle contains at least one edge
which does not appear in any later cycle in the basis (some authors reverse this order).

Weakly fundamental cycle bases have the following useful property, that we can trivialize the fundamental group by attaching $2$-cells to each cycle in the basis.  This is stronger than what we might have for an arbitrary cycle basis, where attaching such $2$-cells might only kill first homology without trivializing the fundamental group.
\begin{lemma}
Given a connected graph $G$, regarded as a $1$-complex, any weakly fundamental cycle basis is actually a free basis for $\pi_1(G)$. As a consequence attaching $2$-cells to this basis produces a simply connected result.
\begin{proof}
Let the basis have $k$ cycles.
Order the cycles $C_1,\ldots,C_k$ of the basis so that each $C_i$ contains some edge $e_i$ not in any later cycle.
$G\setminus (e_1 \cup \ldots \cup e_k)$ must be a tree; if not, the $\{C_i\}$ do not even span $H_1(G;\Z_2)$.
Crushing this tree to a point, $G$ becomes a bouquet of circles, the circles bijective with the edges $\{e_i\}$, exhibiting
$\{e_i\}$ as a free basis.
\end{proof}
\end{lemma}

Our main result is the following {\bf decongestion lemma}.  We do not need the results in the second paragraph of this lemma in the rest of this paper, but we give them for completeness.
This lemma shows the existence of a weakly fundamental cycle basis in which each edge appears only polylogarithmically many times.  Other than MacLane's ``$2$-basis", the property of bounding the number of times an edge appears in a cycle basis does not seem to have been considered much before.  \cite{reich2014cycle} called a basis in which each edge appears at most $p$ times a ``$p$-basis", but we have not found any upper bounds on $p$ in the literature.

\begin{lemma}
\label{dcl}
Given any graph (possibly with multi-edges and self-edges) on $V$ vertices and $E$ edges, with the degree of each vertex bounded by $O(1)$, there is a weakly fundamental cycle basis 
in which each edge appears in at most $O(\log(V)^2)$ cycles in the basis.

Further, there is an efficient randomized algorithm, given in the proof of this lemma, to construct such a basis.  Further, if we assign non-negative weights to each edge of the graph, then with probability $\Omega(1)$ the algorithm returns a cycle basis of total weight bounded by $(\log(V))$ times the sum of edge weights of the graph.

Further, suppose the graph is sparse, and say that two cycles in the basis intersect if they share at least one vertex.  Then, each cycle in the constructed basis intersects at most polylogarithmically many cycles later in the basis.
\begin{proof}
The proof is based on the following algorithm $A$.

The algorithm $A$ is recursive and randomized.  We write $A(G)$ to denote the cycle basis returned given graph $G$ as input.
If $G$ has no edges, then $A(G)=\emptyset$.  Otherwise:
\begin{itemize}
\item[1.] If graph $G$ has at least one vertex with degree $1$, let $v$ be an arbitrary such degree-$1$ vertex and define $G'$ to the be graph with that vertex removed (i.e., the subgraph induced by the set of vertices other than $v$).  Return $A(G')$.

\item[2.] Else, if graph $G$ has at least one vertex $v$ with degree $2$, let $v$ be an arbitrary such degree-$2$ vertex.
We consider two cases.
{\bf Case A:} If $v$ has a self-edge $e$ , let $C$ be the cycle consisting just of that self-edge.  Let $G'$ be the graph obtained by removing edge $e$ from $G$, and return $\{C\} \cup A(G')$. {\bf Case B:}
If $v$ does not have a self-edge, then $v$ has edges to two other vertices $x,y$.  Define $G'$ to be the graph given by removing $v$ from $G$ and adding an edge $(x,y)$.  Compute $A(G')$.  Then, for each cycle in $A(G')$, replace every occurrence of edge $(x,y)$ with $(x,v),(v,y)$ and return this as $A(G)$.

\item[3.] Else, find a simple cycle $C$ of length at most $O(\log(V))$ in $G$.  Let $G'$ be the graph obtained from $G$ by removing an edge of that cycle, choosing that edge uniformly at random.  Return $\{C\} \cup A(G')$.
\end{itemize}

Let us prove first that this algorithm is correct, in that it always returns a weakly fundamental cycle basis for $G$, and let us prove that indeed a cycle of length $O(\log(V))$ does exist in case 3.  Then we use this algorithm to prove the existence of a cycle basis where each edge appears at most $O(\log(V)\log(E))$ times.

The proof that the algorithm returns a cycle basis is inductive: we assume it is correct for all graphs with $V'<V$ vertices and for all graphs on $V$ vertices with $E'<E$ edges, and use that to prove it holds for all graphs with $V$ vertices and $E$ edges.
The base cases $V=0$ or $V>0,E=0$ are trivial.
Under this inductive assumption, it is immediate that the algorithm returns a cycle basis.
Note that if item 1 holds, then $v$ does not participate in any cycle, and so a cycle basis for $G'$ gives a cycle basis for $G$.
The cycle basis the algorithm returns is weakly fundamental since each time in case 2A or 3 that we add a cycle $C$ to the basis, we remove an edge in $C$.

To bound the length of the simple cycle in item 3, note that by assumption all vertices have degree at least $3$.  Hence, starting from any given vertex there are $>2^l$ paths of length $l$ starting at that vertex so for $l\sim \log(V)$ there must be two different paths with the same endpoint, implying a simple cycle of length $O(\log(V))$.

Now we estimate how many times some edge $e$ appears in a cycle basis returned by $A(G)$.
The algorithm is recursive, starting with graph $G$ and defining a sequence of graphs $G',G'',G''',\ldots$.
Let $G_0=G,G_1=G',G_2=G'',\ldots$.
Note that when the algorithm constructs $G_{j+1}$ from $G_j$, the edge set of $G_{j+1}$ is some subset of the edges of $G_j$, possibly union with an extra edge in case {\bf 2B}.
We say that an edge $f$ in $G_{j+1}$ is the {\it child} of an edge $e$ in $G_j$ if $f$ is the same as edge $e$ or if $f$ is obtained in case 2B, where $e=(x,v)$ or $e=(v,y)$ and $f=(x,y)$.
We say that an edge $f$ in $G_j$ is the {\it descendant} of an edge $e$ in $G=G_0$
if there is a sequence $e_0=e,e_1,e_2,\ldots,e_j=f$ with $e_{k+1}$ the child of $e_k$ for $k=0,1,\ldots,j-1$; note that this descendant is unique, if it exists.  Not every edge $e$ has a descendant in $G_j$: some edges are removed.
We say that an edge $e\in G$ is removed on step $j$ if $e$ has a descendant in $G_{j-1}$ but does not have a descendant in $G_j$.

Then, for any edge $e\in G$, the number of times that $e$ appears in the cycle basis is equal to the number of times that the algorithm constructs a cycle $C$ containing a descendant of $e$.
If on the $j$-th step, for some edge $e$ we construct a cycle containing a descendant of $e$, using case 2A of the algorithm,
then $e$ is removed on step $j$.  Hence,
the number of times that $e$ appears in the cycle basis is bounded by one plus the number of times that the algorithm uses case 3 and constructs a cycle $C$ containing a descendant of $e$.

However, each time case 3 occurs, the descendant is removed with probability $\Omega(1/\log(V))$.
Hence, the probability that an edge $e$ appears in at least $w$ such cycles is bounded by
$$(1-\Omega(1/\log(V))^w.$$  For some $w=O(\log(V) \log(E))$, this probability can be made strictly smaller than $1/2E$, and so by a union bound, no such edge exists with probability $\geq 1/2$.

Suppose the input graph $G$ has no multi-edges or self-edges (this is the case that is relevant for our applications).  Then, $E=O(V^2)$; this is the worst case for a dense graph, though for our applications with a sparse $G$ we have $E=O(V)$.
So $\log(V) \log(E)=O(\log(V)^2)$.

If $G$ has multi-edges or self-edges, we can still prove the claim of the lemma by applying a ``pre-processing" before running the algorithm.  This case is not needed for us, but we give it for completeness.  First, remove self-edges and replace every multi-edge with a single edge to obtain some $\tilde G$ with $V$ vertices and $\tilde E=O(V^2)$ edges.  Construct a cycle basis for $\tilde G$ in which every edge appears at most $O(\log(V)^2)$ times.
If an edge in $\tilde G$ connects two vertices which have multiple edges between them in $G$, replace each occurence of that edge in the cycle basis with an arbitrary one of those multiple edges in $G$.
Then, add every self-edge of $G$ as a cycle to this cycle basis.  Also, if any two vertices $u,v$ in $G$ have multiple edges, labeled $e_1,e_2,\ldots,e_k$ for some $k$, add the cycle $e_j e_{j+1}^{-1}$ for  each $j=1,\ldots,k-1$ to this cycle basis.  The result is a cycle basis for $G$ with the given property.

The algorithm $A$ that we have given is efficient, since it terminates after at most $E$ recursive calls and each step can be done efficiently (one can efficiently find a shortest cycle in the graph by, for example, for each edge in turn, considering the graph with that edge removed and searching for a shortest path between the vertices which were endpoints of the edge).  Each time the algorithm is run, it succeeds in finding such a basis with probability $\geq 1/2$, and the desired properties of the basis can be efficiently verified.

Suppose non-negative edge weights are assigned to each edge.  Then expected cycle basis weight is, by linearity of expectation, the sum over edges of the weight of that edge times the expected number of times the algorithm constructs a cycle $C$ containing a descendant of the edge.  That expected number of times is $O(\log(V))$.  Hence, the expected cycle basis weight is bounded by the total edge weight times $O(\log(V))$, so with probability $\Omega(1)$, the cycle basis weight is bounded by total edge weight times $O(\log(V))$.

Finally, suppose the input graph $G$ is sparse.  Then, the graph remains sparse throughout the algorithm as the degree only decreases.  When the algorithm constructs a cycle in either case 2A or 3 when called with some graph $G'$ (where perhaps $G'=G$ or perhaps $G'$ is constructed by the recursive process), the cycle has length at most logarithmic in the number of vertices of $G'$.  All subsequent cycles constructed contain each edge (and, by sparseness, each vertex) of $G'$ at most polylogarithmically many times.  So, the cycle intersects only polylogarithmically many cycles later in the basis.
\end{proof}
\end{lemma}

\section{Pushing to Boundary: Alternative Proof}
\label{ptbap}
Here we give an alternative to lemma \ref{push} that may be of independent interest as an isoperimetric inequality.

\begin{lemma}
\label{pushalt}
    Let $(B^k,S^{k-1})$ be the unit ball in Euclidean $k$-space. Let $V$ be a smooth singular $p$-cycle in $S^{k-1}$ for $p < k-1$. Suppose $V$ is the boundary of a smooth singular chain $W$ lying in $B^k$, then $V$ is also the boundary of a smooth singular chain $\lbar{W}$ lying in $S^{k-1}$ with
    \[
        p\parea(\lbar{W}) \leq 2^p(1+2\pi)p\parea(W)
    \]
\end{lemma}

\begin{proof}
    By the compactness properties of varifolds \cite{federer}, $W$ may be replaced with a least area varifold $Z$, $\de Z = V$, and $p\parea(Z) \leq p\parea(W)$. By proposition 3.7 of \cite{cm11}, the monotonicity formula holds in the context of stationary varifolds. This formula tells us that if we divide $Z$ into $Z_{\leq 1/2} \cup Z_{>1/2}$, the portions of $Z$ inside and outside $B^k_{1/2}$, the ball of radius $\frac{1}{2}$, that
    \begin{equation}\label{eq:lmproof1}
        p\parea(Z_{>1/2}) > \left(1 - \left(\frac{1}{2}\right)^p\right) p\parea(Z)
    \end{equation}
    
    Furthermore the co-area formula tells us that
    \begin{equation}\label{eq:lmproof2}
        dp\parea(Z_{>1/2}) \geq \int_{\tau=1/2}^1 (p-1)\parea(Z \cap S_\tau^{d-1})\ d\tau
    \end{equation}
    $S_{\tau}^{d-1}$ the sphere of radius $\tau$.

    Lines \ref{eq:lmproof1} and \ref{eq:lmproof2} imply that for some $\tau \in (\frac{1}{2},1)$ $(p-1)$-area of $Z \cap S_{\tau^{d-1}}$ satisfies:
    \begin{equation}
        (p-1)\parea(Z \cap S_\tau^{d-1}) \leq 2\left(1 - \left(\frac{1}{2}\right)^p\right)p\parea(Z) < 2p\parea(Z)
    \end{equation}

    Now let $C$ be the cone $C(Z \cap S_\tau^{d-1})$ embedded in $S_\tau^{d-1}$ as follows. Carefully, as described next, choose a point $\lbar{q} \in S_\tau^{d-1}$ disjoint from $Z$. Let $C$ be the geodesic cone where each point of $Z \cap S_\tau^{d-1}$ is joined to $q$, the antipode of $\lbar{q}$, by the unique shortest geodesic arc (which has length $<\pi\tau$).

    \begin{customlm}{A}
        Given any piecewise smooth $j$-cycle and an appropriate choice of a point $q \in S_\tau^{d-1}$, then $p\parea(C) < \pi\tau((p-1)\parea(Z \cap S^{d-1})) \leq 2\pi\tau p\parea(Z)$.
    \end{customlm}

    We will need a small digression into integral geometry to prove Lemma A. For motivation recall Croften's Theorem, that the length of a rectifiable curve $\gamma \in \R^n$ is (up to a closure of a constant ) the integral over the kinematic measure\footnote{See \cite{santalo76}.} on $\{l\}$, the lines of $\R^n$ of the number of intersection points, $\abs{l \cap \gamma}$
    \begin{equation}
        L(\gamma) :: \int_{\{l\}} \abs{l \cap \gamma}
    \end{equation}

    An equivalent formulation is that
    \begin{equation}\label{eq:dualform}
        L(\gamma) :: \int_{r \in \R^n} d(\text{vol}) L(\gamma_r)
    \end{equation}
    where $\gamma_r$ is the projection of $\gamma$ into the visual sphere of oriented rays emanating from $r \in \R^n$ and the length in the integral is w.r.t\ spherical geometry of $S^{n-1}$.

    In spherical geometry there is also the notion of the visual $S_s^{n-1}$ sphere from each point $s \in S^n$, defined by the oriented geodesic arcs of length $\pi$ leaving $s$. Very similar to (\ref{eq:dualform}) is

    \begin{customlm}{B}
        Let $\gamma$ be a piecewise smooth singular $j$-cycle in $S^n$, $j < n$, then
        \[
            j\parea(\gamma) = \int_{s \in S^n} d(\text{vol}) j\parea(\gamma_s)
        \]
        (actually the constant of proportionality is one in this case).
    \end{customlm}

    \begin{proof}
        To develop some intuition, consider the appearance of a small round $j$-disk $\Delta$ of radius $= \delta > 0$ whose center is at some distance $\Delta$ from $s$. One may check that the \emph{apparent area}, i.e.\ the area in the visual sphere about $s$, is
        \begin{equation}\label{eq:visualsphere}
            \operatorname{A}_s(\Delta) = \Omega_j \left(\prod_{i=1}^j \sin(\beta_i)\right) \slash \sin(t)
        \end{equation}
        where $\Omega_j$ is the $j$-volume of the unit $j$-ball, and where the $\beta_i$ are the $j$ dihedral angle between $\Delta$ and the ray from $s$ to the origin of $\Delta$.

        The important thing to notice is that $\sin(t)$ in the denominator, objects very close to distance 0 \emph{or} distance $\pi$ look enormous.

        We view every small speck of $\gamma$, which we may think of as a tiny disk in a smooth singular simplex of $\gamma$, so small that it is indeed well-approximated by a round $j\text{-ball}_\delta$ such as $\Delta$. So, it is very easy to believe that the only thing that, on average, could affect $\int_{s \in S^n} \operatorname{A}_s(\gamma)$ is the area $\operatorname{A}(\gamma)$. This can be established by elementary approximation arguments and observing that both sides of (\ref{eq:visualsphere}) are additive under gluing. In fact, for every $s$
        \begin{equation}\label{eq:additive}
            \operatorname{A}_s(X \cup Y) = \operatorname{A}_s(X) + \operatorname{A}_s(Y) - \operatorname{A}_s(X \cap Y)
        \end{equation}
        since these are just ordinary $j$-areas in $S^{n-1}$.

        To see that the constant is indeed unity for equation \ref{eq:additive}, notice that if $S^j \subset S^n$ is any totally geodesic (``great'') $j$-sphere in the $n$-sphere that from any point $s \in S^n$ (disjoint from $S^j$, or on $S^j$) the visual image of $S^j$ in the visual sphere $S_s^{n-1}$ is also a great $S^j \subset S_s^{n-1}$.\footnote{One way to see this is to consider the unique great $S^{j+1}$ containing $s$ and $S^j$. The visual image of $S^j$ in this $S^{j+1}$ is the total field (with multiplicity 1) and equal to the visual image of $S^j$ from $s$ in $S_s^{n-1}$, in both cases it has exactly the same $j$-area as the source cycle $S^j$.} This example shows that the constant is one: for a great $S^j$, not just the average, but every individual visual image has the same area as the original $j$-cycle. This completes the proof of Lemma B.
    \end{proof}

    From Lemma B we can now prove an isoperimetric inequality, Lemma C, whose statement and proof result from discussions with Xin Zhou. Lemma A is the immediate application of Lemma C.

    \begin{customlm}{C}
        Let $\gamma$ be a piecewise smooth singular $j$-cycle $\gamma \subset S^n$. There is a point $q \in S^n$ so that the core $C_q(\gamma)$ the cone of geodesic segments of length $\leq \pi$ beginning on $\gamma$ and ending at $a$ satisfies:
        \[
            (j+1)\parea(C_q(\gamma)) \leq \pi(j\parea(\gamma))
        \]
    \end{customlm}

    \begin{proof}
        Since $\int_{s \in S^n} \operatorname{A}(\gamma_s) = \operatorname{A}(\gamma)$, where we have normalized so that $\int_{S^n} 1 = 1$, there exists point $q \in S^n$ so that $\operatorname{A}(\gamma_q) \leq \operatorname{A}(\gamma)$. Now $C_q(\gamma)$ contains a geodesic segment of length $\leq \pi$ for every point of $\gamma$. In the Euclidean case Fubini's then would conclude the proof, but the positive curvature, bending the fibers \emph{together} reduces the area of the cone making the inequality hold even more strongly.
    \end{proof}

    To translate Lemma B to Lemma C set $n = d-1$ and $j = q-1$.

    Now consider the verifold $Z^\pr$, with $\de Z^\pr = W$, $Z^\pr = Z_{>1/2} \cup C_q$. We have
    \begin{equation}
        p\parea(Z^\pr) \leq \left(1 - \left(\frac{1}{2}\right)^p\right)p\parea(Z) + 2\pi\tau(p\parea(Z)) < (1+2\pi)p\parea(Z)
    \end{equation}

    Now define $\lbar{Z}$ to be the radially projected image of $Z^\pr$ into $S^{d-1}$. Since this projection has Lipschitz constant $=2$, we obtain:
    \begin{equation}
        p\parea(\lbar{Z}) \leq 2^p(1+2\pi) p\parea(Z)
    \end{equation}
    as desired.
\end{proof}

\end{document}